\documentclass[reqno]{amsart}

\usepackage{amsthm}
\usepackage[usenames,dvipsnames]{pstricks}
\usepackage{epsfig}
\usepackage{pst-grad} 
\usepackage{pst-plot} 
\usepackage{graphicx,subfigure}
\usepackage{amsmath,amssymb}

\usepackage{acronym}
\acrodef{BO}{{\sl Benjamin-Ono}}
\acrodef{rBO}{{\sl regularized Benjamin-Ono}}
\acrodef{rILW}{{\sl regularized Intermediate Long Wave}}
\acrodef{DSW}{{\sl Dispersive Shock Wave}}
\acrodef{DSWs}{{\sl Dispersive Shock Waves}}
\acrodef{ILW}{{\sl Intermediate Long Wave}}
\acrodef{CGN}{{\sl Conjugate Gradient-Newton}}
\acrodef{SW/SW}{{\sl Shallow water / Shallow water}}
\acrodef{B/B}{{\sl Boussinesq / Boussinesq}}

\makeatletter
\newcommand{\sech}{\mathop{\operator@font sech}}
\newcommand{\sign}{\mathop{\operator@font sign}}
\makeatother

\newtheorem{theorem}{Theorem}[section]

\newtheorem{remark}{Remark}[section]

\numberwithin{equation}{section}

\begin{document}

\title[]{Long-term behaviour of symmetric partitioned linear multistep methods II. Invariants error analysis for some nonlinear dispersive wave models}

\author{Bego\~{n}a Cano}
\address{\textbf{B.~Cano:} Applied Mathematics Department, University of Valladolid, P/ Belen 15, 47011, Valladolid, Spain}
\email{bcano@uva.es}

\author{Angel Dur\'an}
\address{\textbf{A.~Dur\'an:} Applied Mathematics Department, University of Valladolid, P/ Belen 15, 47011, Valladolid, Spain}
\email{angeldm@uva.es}

\author{Melquiades Rodr\'\i guez}
\address{\textbf{M.~Rodr\'\i quez:} Applied Mathematics Department, University of Valladolid, P/ Belen 15, 47011, Valladolid, Spain}
\email{melquiades.rodriguez@uva.es}



\subjclass[2010]{XXXX}




\begin{abstract}
In this paper, the use of partitioned linear multistep methods (PLMM) as time integrators for the numerical approximation  of some partial differential equations (pdes) is studied. We consider the periodic initial-value problem of two nonlinear dispersive wave models as case studies. From the spatial discretization with pseudospectral methods, the theory developed for PLMMs by the authors in a previous companion paper is applied to analyze the time integration with PLMMs of the semidiscrete equations when approximating solitary wave solutions. The results are illustrated with some numerical experiments. In addition, a computational study is performed in an exploratory fashion to analyze the extension of the results to the approximation of more general localized solutions.
\end{abstract}

\maketitle

\section{Introduction}\label{sec1}
In a previous paper, \cite{CDR1}, the authors analyzed the growth with time of the error of the numerical approximation with PLMM of initial-value ode problems of dimension $n \ge 2$, written in the form
\begin{eqnarray}
\dot{P}(t)&=& f(P(t),Q(t)), \nonumber \\
\dot{Q}(t)&=& g(P(t),Q(t)), \quad t_{0}\leq t\leq T,\label{system} \\
P(t_0)&=&P_0, \nonumber \\
Q(t_0)&=&Q_0, \nonumber
\end{eqnarray}
where $T>t_{0}$, $P=P(t)\in \mathbb{R}^d, Q=Q(t)\in \mathbb{R}^{n-d},  1 \leq d \leq n-1$,   $f=f(P,Q)$ and $g=g(P,Q)$ smooth. The PLMMs are defined from some irreducible pair of generating polynomials $(\rho_p, \sigma_p)$ and $(\rho_q, \sigma_q)$, as
\begin{eqnarray}
\rho_p(E) P_n&=&\Delta t \sigma_p (E) f(P_n,Q_n), \nonumber \\
\rho_q(E) Q_n&=& \Delta t \sigma_q (E) g(P_n,Q_n),\quad \Delta t>0, \label{eq12}
\end{eqnarray}
where $(P_{n},Q_{n})$ approximates the solution $(P(t_{n}),Q(t_{n}))$ of (\ref{system}) at $t_{n}=t_{0}+n\Delta t, n=0,\ldots,N, N\Delta t=T-t_{0}$. Both schemes (\ref{eq12}) are assumed to be zero stable and have the same order $r$, with a suitable procedure for the starting values $P_\nu$ ($\nu=0,1,\dots,k_p-1$), $Q_\nu$ ($\nu=0,1,\dots,k_q-1$) satisfying
\begin{eqnarray}
P(t_\nu)-P_\nu&=&O(\Delta t^{r}), \quad \nu=0,1,\dots,k_p-1,\nonumber\\
Q(t_\nu)-Q_\nu&=&O(\Delta t^{r}),\quad \nu=0,1,\dots,k_q-1,\label{bcad0a}
\end{eqnarray}
where $\mbox{deg}(\rho_p)=k_p$, $\mbox{deg}(\rho_q)=k_q$. The main purpose of \cite{CDR1} was to study the time behaviour of the error with respect to both the solutions and conserved quantities of (\ref{system}) when using PLMMs of the form (\ref{eq12}) as numerical integrators. The key to this end is the derivation of an asymptotic expansion of the global error in powers of the step size $\Delta t$. Thus, the time behaviour of the errors in the conserved quantites can be analyzed from the growth with time of the corresponding coefficients of the expansion. The main conclusions of the analysis are the following:
\begin{itemize}
\item PLMMs where the first characteristic polynomials $\rho_{p}, \rho_{q}$ have common roots of unit modulus different from $x_{1}=1$ may lead to an exponential growth with time of the error and they are not recommended in general.
\item When the PLMM is symmetric (that is, when both $(\rho_p,\sigma_p)$ and $(\rho_q,\sigma_q)$ are symmetric) and with no common unitary roots of $\rho_{p}, \rho_{q}$ except $x_{1}=1$, the asymptotic expansion can be simplified, providing the conditions to control the time behaviour of the errors in the invariants. In this sense, the additional fact that these symmetric PLMMs can be explicit makes them an alternative as numerical schemes from the point of view of the geometric integration, especially for long-term simulations.
\end{itemize}
The main purpose of the present paper, as a second part of the approach, is to analyze if these geometric properties of this type of symmetric PLMMs can be used in the numerical approximation of nonlinear dispersive wave phenomena, modelled by partial differential systems and with special localized solutions as essential elements of their dynamics. The ambitious goal will be here delimited considering, as case studies, two types of equations:
\begin{itemize}
\item The nonlinear Schr\"{o}dinger (NLS) equation
\begin{eqnarray}
i u_t(x,t)+u_{xx}(x,t)+|u|^{2\sigma} (x,t) u(x,t)&=&0, \quad x\in \mathbb{R}, t>0,\label{nls}
\end{eqnarray}
with $\sigma>0$.
\item The $(a,b,c,d)$-Boussinesq system
\begin{eqnarray}
&&\eta_{t}(x,t)+w_{x}(x,t)+(w \eta)_{x}(x,t)+a w_{xxx}(x,t)-b \eta_{xxt}(x,t)=0, \label{bb} \\
&&w_{t}(x,t)+\eta_{x}(x,t)+w(x,t) w_{x}(x,t)+c \eta_{xxx}(x,t)-d w_{xxt}(x,t)=0, \nonumber
\end{eqnarray}
for $x\in \mathbb{R}, t>0$, and real coefficients $a,b,c,d$ where $a,c\leq 0, b,d \geq 0$, and $a+b+c+d=1/3$.
\end{itemize}
The origin of (\ref{nls}) and its influence in different areas of nonlinear science are well known, cf. e.~g. \cite{SulemS1999} and references therein. On the other hand, the system (\ref{bb}) is derived and analyzed in \cite{BCS1,BCS2} as one-dimensional model for the two-way propagation of surface waves on a uniform, horizontal channel approximating the corresponding Euler equations. In (\ref{bb}), $\eta=\eta(x,t)$ represents the value at $(x,t)$ of the height of the free surface of the incompressible, inviscid fluid of irrotational flow with respect to some level of rest, while $w=w(x,t)$ is proportional to the value at $(x,t)$ of the horizontal component of the velocity of the fluid at some height with respect to the bottom. (Cross-channel variations are assumed negligible.) The system admits a corresponding version to describe the propagation of internal waves, see e.~g. \cite{BLS2008,Du,Saut}.

The present paper is focused on the use of symmetric PLMMs when approximating the dynamics of solitary wave solutions of (\ref{nls}) and (\ref{bb}). They are localized, traveling wave solutions with a permanent profile and constant speed that decays to zero at infinity. In the case of NLS equation (\ref{nls}), the solitary waves are known explicitly, see \cite{D} and Section \ref{sec2} of the present paper. Boussinesq systems (\ref{bb}) also admit solitary wave solutions, although their exact formula is known only for specific speeds. In both cases, the solitary waves decay exponentially to zero as $x\rightarrow \pm\infty$. Our approach is focused on these solutions due to mainly two reasons: one is their influence in the general dynamics of these and other nonlinear dispersive wave models, \cite{Bona}. The second one, somehow related to the first, is that, in the Hamiltonian case, the solitary waves can be seen as relative equilibria or critical points of the Hamiltonian constrained to fixed values of those invariants associated to the symmetry groups involving their generation (such as spatial translations or rotations). This property suggests a relation between these waves and the conserved quantities which should be taken into account in the simulation of the dynamics, \cite{D}.

In order to analyze the behaviour of symmetric PLMMs when approximating solitary wave solutions of (\ref{nls}) and (\ref{bb}), the following approach will be made. Due to the localized property of the solutions, the initial-value problem (ivp) for both will be approximated by a periodic ivp on a long enough interval. This is a standard strategy to study computationally the dynamics of solitary waves, justified in some cases, \cite{BChen,Pasciak1982}. Then the corresponding periodic problem will be discretized in space with a Fourier spectral collocation method. The resulting semidiscrete ode system will be finally integrated in time with a symmetric PLMM. The description of this procedure for each equation (\ref{nls}) and (\ref{bb}), and the corresponding analysis of the time behaviour of the errors in the invariants, are made in Section \ref{sec2}.

The performance of the PLMMs will be analyzed in this paper from the time behaviour of the invariants of the problems through the numerical integration. Note that in the cubic case ($\sigma=1$) the equation (\ref{nls}) is integrable, and admits an infinite number of conserved quantities; the study of the present paper will be focused on the time behaviour of the error with respect to semidiscrete versions of the first three of them, mass, momentum, and energy, which are related to the formation of solitary wave solutions. As far as (\ref{bb}) is concerned, the systems do not seem to be integrable, \cite{DougalisM2008}, but when $b=d$ they are Hamiltonian and admit an additional second invariant, of quadratic type, determining the solitary wave solutions as relative equlibria as well. Semidiscrete forms of this quantity and the Hamiltonian are derived, and the analysis of the growth with time of the corresponding errors are made.

The  main contributions in this sense are the following. Assume that a symmetric $r$th-order PLMM, where $\rho_{p}, \rho_{q}$ have no common roots of unit modulus different from $x_{1}=1$, is used as time integrator for the semidiscrete systems of (\ref{nls}) and (\ref{bb}). Under these assumptions, and when approximating the corresponding solitary waves, it holds that:
\begin{itemize}
\item If the starting values of the symmetric PLMM procedure satisfy (\ref{bcad0a}), then the errors in the invariants are $O(\Delta t^{r})$ with a Landau constant independent of time for $t-t_0=O(\Delta t^{-1})$.
\item In the case of solitary waves of NLS equation, if the starting values differ from the exact ones in $O(\Delta t^{r+1})$ terms, then the errors in the invariants are $O(\Delta t^{r})$ with a Landau constant independent of time for $t-t_0=O(\Delta t^{-2})$.
\item In the case of solitary waves of Boussinesq systems, if the starting values differ from the exact ones in $O(\Delta t^{r+1})$ terms, then the errors in the invariants are $O(\Delta t^{r+1})$ with a Landau constant independent of time for $t-t_0=O(\Delta t^{-1})$.
\item In the case of solitary waves of Boussinesq systems, if the starting values differ from the exact ones in $O(\Delta t^{r+2})$, the errors in the invariants are $O(\Delta t^{r+2})$ for $t-t_0=O(\Delta t^{-1})$.
\end{itemize}
Some consequences of these results are explored in a computational study in Section \ref{sec3}. Here the aim is two-fold. We first illustrate the conclusions with some numerical experiments of comparison with other LMMs, one partitioned but not symmetric and one nonpartitioned and nonsymmetric, \cite{CDR1}. Then, a possible extension of the results is studied numerically. This concerns the performance when approximating other localized solutions whose dynamics involve solitary waves. The experiments are focused on the evolution from small perturbations of solitary waves and the resolution of initial data into trains of solitary waves (the so-called resolution property). It is important to determine the influence of the conservation of the invariants in the numerical integration on the correct simulation of relevant parameters of the waves, such as the amplitude, speed, and phases, \cite{D}. The analysis of this point when integrating with symmetric PLMM's, for the two problems considered in the present paper will be the subject of a future work.

%
%

\section{Analysis of long-term integration of some nonlinear dispersive models}
\label{sec2}

\subsection{Preliminaries}
\label{secprel}
We recall in this subsection some particular details on the asymptotic expansion of the global error when integrating (\ref{system}) with the PLMMs which have been described in the introduction. These will be required for the analysis of the discretization of nonlinear dispersive models.

According to Theorem 2.1 in \cite{CDR1}, when the PLMM is such that $1$ is the only common unitary root of the first characteristic polynomials, there exist functions $E_{j,1,\alpha}$, $E_{j,i,\alpha,\beta}$, with $\alpha,\beta \in \{p,q\}$ such that, as $\Delta t\to 0$,
\begin{eqnarray}
P_n-P(t_n)&=&\sum_{j=r}^{2r-1} {\Delta t}^j \big[E_{j,1,p}(t_n) +\sum_{i=2}^{k_p'} x_{i,p}^n E_{j,i,pp}(t_n)+ \sum_{i=2}^{k_q'} x_{i,q}^n E_{j,i,pq}(t_n)\big]
\nonumber \\&&
+O({\Delta t}^{2r}), \nonumber \\
Q_n-Q(t_n)&=&\sum_{j=r}^{2r-1} {\Delta t}^j \big[E_{j,1,q}(t_n) +\sum_{i=2}^{k_p'} x_{i,p}^n E_{j,i,qp}(t_n)+ \sum_{i=2}^{k_q'} x_{i,q}^n E_{j,i,qq}(t_n)\big]
 \nonumber \\&&
+O({\Delta t}^{2r}). \label{ae}
\end{eqnarray}
where $\{x_{i,p} \}_{i=2}^{k_p'}$, (resp. $\{x_{i,q} \}_{i=2}^{k_q'}$) are the unitary roots of $\rho_p$ (resp. $\rho_q$) different from $1$.

Then, if $I$ is a smooth invariant quantity of (\ref{system}), according to Theorem 4.2 in \cite{CDR1},
\begin{eqnarray}
\lefteqn{I\left( \begin{array}{c} P_n \\ Q_n \end{array} \right)- I\left( \begin{array}{c} P_0 \\ Q_0 \end{array} \right)} \nonumber \\
&=& \sum_{j=r}^{2r-1} \Delta t^j \nabla I \left( \begin{array}{c} P(t_n) \\ Q(t_n) \end{array} \right)^T \bigg[  \sum_{i=2}^{k_p'} x_{i,p}^n \left( \begin{array}{c} E_{j,i,pp}(t_n) \\ E_{j,i,qp}(t_n) \end{array} \right)+ \sum_{i=2}^{k_q'} x_{i,q}^n \left( \begin{array}{c} E_{j,i,pq}(t_n) \\ E_{j,i,qq}(t_n) \end{array} \right) \bigg] \nonumber \\
&&- \sum_{j=r}^{2r-1} \Delta t^j \int_{t_0}^{t_n} (\nabla I) \left( \begin{array}{c} P(s) \\ Q(s) \end{array} \right)^T \left( \begin{array}{c} c_{j,p} P^{(j+1)}(s) \\ c_{j,q} Q^{(j+1)} (s) \end{array} \right)ds \nonumber \\
&&+\sum_{j=r}^{2r-1} \Delta t^j \nabla I \left( \begin{array}{c} P_0 \\ Q_0 \end{array} \right)^T  \left( \begin{array}{c} E_{j,1,p}(t_0) \\ E_{j,1,q}(t_0) \end{array} \right)+ O({\Delta t}^{2r}), \label{thI}
\end{eqnarray}
where $\{c_{j,p}\}, \{c_{j,q}\}$ are the coefficients associated to the local truncation error of both methods. More precisely,
\begin{eqnarray}
\rho_p(E) Y(t_n)- \Delta t \sigma_p(E) \dot{Y}(t_n)= \sigma_p(E) \big( \sum_{j=r}^{2r-1} c_{j,p} {\Delta t}^{j+1} Y^{(j+1)}(t_n)\big) +O({\Delta t}^{2r+1}), \nonumber \\
\rho_q(E) Y(t_n)-\Delta t \sigma_q(E) \dot{Y}(t_n)= \sigma_q(E) \big( \sum_{j=r}^{2r-1} c_{j,q} {\Delta t}^{j+1} Y^{(j+1)}(t_n)\big)+O({\Delta t}^{2r+1}). \nonumber
\end{eqnarray}
As also stated in Remark 4.2 of \cite{CDR1}, since $x_{i,p}^n$ and $x_{i,q}^n$ have unit modulus, the growth with time of error in the invariant is determined by that corresponding to the smooth part of the numerical solution (that associated to the integral third line in (\ref{thI})) and
\begin{eqnarray}
\nabla I \left( \begin{array}{c} P(t_n) \\ Q(t_n) \end{array} \right)^T  \left( \begin{array}{c} E_{j,i,pp}(t_n) \\ E_{j,i,qp}(t_n) \end{array} \right), \quad \nabla I \left( \begin{array}{c} P(t_n) \\ Q(t_n) \end{array} \right)^T  \left( \begin{array}{c} E_{j,i,pq}(t_n) \\ E_{j,i,qq}(t_n) \end{array} \right). \label{inc}
\end{eqnarray}

In the particular case that (\ref{system}) is Hamiltonian and $I=H$ is the Hamiltonian of the system, according to Corollary 4.1 in \cite{CDR1}, the third line in (\ref{thI}) can be written as
\begin{eqnarray}
\sum_{j=r}^{2r-1} \Delta t^j \int_{t_0}^{t_n} [c_{j,q} \dot{P}(t)^T Q^{(j+1)}(t)-c_{j,p}\dot{Q}(t)^T P^{(j+1)}(t)]dt,
\label{secline}
\end{eqnarray}
which if, in addition, both methods are symmetric (and thus $c_{j,p}=c_{j,q}=0$ for odd $j$ \cite{S}) becomes
\begin{eqnarray}
\lefteqn{\sum_{k=\frac{r}{2}}^{r-1} \Delta t^{2k} \int_{t_0}^{t_n} [c_{2k,q} \dot{P}(t)^T Q^{(2k+1)}(t)-c_{2k,p}\dot{Q}(t)^T P^{(2k+1)}(t)]dt} \nonumber \\
&=&\sum_{k=r/2}^{r-1} {\Delta t}^{2k} \bigg[c_{2k,q}\sum_{l=1}^k (-1)^{l+1}[P^{(l)}(t)^T Q^{(2k+1-l)}(t)-P^{(l)}(t_0)^T Q^{(2k+1-l)}(t_0)] \nonumber \\
&& \hspace{1.8cm} -c_{2k,p}\sum_{l=1}^k (-1)^{l+1}[Q^{(l)}(t)^T P^{(2k+1-l)}(t)-Q^{(l)}(t_0)^T P^{(2k+1-l)}(t_0)] \bigg] \nonumber \\
&&+\sum_{k=r/2}^{r-1} {\Delta t}^{2k} (-1)^{k+2} (c_{2k,q}-c_{2k,p}) \int_{t_0}^t P^{(k+1)}(s)^T Q^{(k+1)}(s)ds . \label{ham_smooth}
\end{eqnarray}

On the other hand, following Theorem 3.1 and Remark 3.1 in \cite{CDR1}, whenever the transition matrices associated to
\begin{eqnarray}
\dot{E}(t)=\lambda_{p,i,p} f_P(P(t),Q(t))E(t), \quad \dot{E}(t)=\lambda_{q,i,q} g_Q(P(t),Q(t))E(t),
\label{tm}
\end{eqnarray}
with
$$
\lambda_{p,i,p}=\frac{\sigma_p(x_{i,p})}{x_{i,p}\rho_p'(x_{i,p})}, \quad \lambda_{q,i,q}=\frac{\sigma_q(x_{i,q})}{x_{i,q}\rho_q'(x_{i,q})},$$
are bounded, as well as
\begin{eqnarray}
f_P(P(t),Q(t)), f_Q(P(t),Q(t)), g_P(P(t),Q(t)), g_Q(P(t),Q(t))
\label{fun_jac}
\end{eqnarray}
and their time derivatives,  the terms of the error in (\ref{ae}) associated to the non-common unitary roots behave as $O(\Delta t^r)$ (where the constant in Landau notation is bounded for $t-t_0=O(\Delta t^{-1})$). Moreover, if the starting procedure has order $r+1$, the result is valid for $t-t_0=O(\Delta t^{-2})$ (cf. Introduction). More particularly, in the last case, the coefficients associated to $j=r$ vanish, those associated to $j=r+1$ are bounded, the ones corresponding to $j=r+2$ grow linearly, etc. Furthermore, we will use afterwards that in the case that the starting procedure has order $r+2$, $E_{r+1,i,pq}$ and $E_{r+1,i,qp}$ vanish and  the initial conditions for the rest of the coefficients in $j=r+1$ satisfy (see (A5) in \cite{CDR1})
\begin{eqnarray}
E_{r+1,1,p}(t_0)+\sum_{i=2}^{k_p'} x_{i,p}^{\nu} E_{r+1,i,pp}(t_0)= \nu c_r^p p^{(r+1)}(t_0), \nonumber \\
E_{r+1,1,q}(t_0)+\sum_{i=2}^{k_q'} x_{i,q}^{\nu} E_{r+1,i,qq}(t_0)= \nu c_r^q q^{(r+1)}(t_0). \label{icrm2}
\end{eqnarray}
Besides, $E_{r+1,i,pp}(t)$ and $E_{r+1,i,qq}(t)$ satisfy each of the equations in (\ref{tm}).

\subsection{Approximation of solitary wave solutions of the nonlinear Schr\"odinger equation}
\label{sec21}
The NLS equation (\ref{nls}) can also be written as a real system
\begin{eqnarray}
p_{t}+q_{xx}+(p^{2}+q^{2})^{\sigma}q&=&0,\nonumber\\
q_{t}-p_{xx}-(p^{2}+q^{2})^{\sigma}p&=&0,\label{nls2}
\end{eqnarray}
where $u=p+iq$. We will consider both the ivp and the periodic ivp of (\ref{nls}), or (\ref{nls2}), and in the last case we will define the problem on an interval of periodicity $(l_i,l_s)$ of period $l=l_s-l_i$. It is well known that the quantities
\begin{eqnarray}
I_{1}(p,q)&=&\frac{1}{2}\int_{\Omega}(p^{2}+q^{2})dx=\frac{1}{2}\int_{\Omega}|u|^{2}dx, \label{nls3a}\\
I_{2}(p,q)&=&\frac{1}{2}\int_{\Omega}(pq_{x}-qp_{x})dx=\frac{1}{2}\int_{\Omega}{\rm Im}(u\overline{u}_{x})dx, \label{nls3b}\\
H(p,q)&=&\int_{\Omega}\left(\frac{1}{2}\left( (p_{x})^{2}+(q_{x})^{2}\right)\right.\left.-\frac{1}{2\sigma+2} (p^{2}+q^{2})^{\sigma+1}\right) d{x},\label{nls3c}
\end{eqnarray}
(called, respectively, mass, momentum and energy) are invariants by smooth enough and localized solutions of the ivp of (\ref{nls}) (when $\Omega=\mathbb{R}$) or smooth and $l$-periodic solutions of the periodic ivp on $\Omega=(l_i,l_s)$. The energy (\ref{nls3c}) is the Hamiltonian function of the canonical Hamiltonian formulation of (\ref{nls2}).

The quantities (\ref{nls3a})-(\ref{nls3c}) play a two-fold role in the existence and formation of solitary wave solutions, cf. e.~g. \cite{D} and references therein. The first one is concerned with the symmetry group, \cite{Olver}, of (\ref{nls2})
$(v,w)\mapsto G_{(\alpha,\beta)}(v,w), \alpha,\beta\in\mathbb{R}$ defined as
\begin{eqnarray}
G_{(\alpha,\beta)}(v,w)(x)=\begin{pmatrix} \cos\alpha&-\sin\alpha\\\sin\alpha&\cos\alpha\end{pmatrix}\begin{pmatrix}v(x-\beta)\\w(x-\beta)\end{pmatrix}.\label{symg}
\end{eqnarray}
This means that if $(p(x,t),q(x,t))$ is a solution of (\ref{nls2}), then $(\widetilde{p},\widetilde{q})= G_{(\alpha,\beta)}(p,q)$ is a solution of (\ref{nls2}). The relation between (\ref{symg}) and (\ref{nls3a})-(\ref{nls3b}) is given in the sense that
\begin{eqnarray*}
\frac{d}{d\alpha}\Big|_{\alpha=0}G_{(\alpha,0)}(v,w)(x)&=&\begin{pmatrix}0&1\\-1&0\end{pmatrix}\delta I_{1}(v,w)(x),\\
\frac{d}{d\beta}\Big|_{\beta=0}G_{(0,\beta)}(v,w)(x)&=&\begin{pmatrix}0&1\\-1&0\end{pmatrix}\delta I_{2}(v,w)(x),
\end{eqnarray*}
where $\delta I$ denotes variational (Fr\'echet) derivative of $I$. That is, (\ref{nls3a})-(\ref{nls3b}) determine the infinitesimal generators of the symmetry group (\ref{symg}) of rotations and spatial translations.

In addition, the solitary waves can be derived as critical points $u_{0}(x)=p_{0}(x)+iq_{0}(x)$ of (\ref{nls3c}) constrained to fixed values of the mass (\ref{nls3a}) and the momentum (\ref{nls3b}). The constrained critical point condition (also called relative equilibrium condition)
\begin{eqnarray}
\delta\left(H(u_{0})-\lambda_{0}^{1}I_{1}(u_{0})-\lambda_{0}^{2}I_{2}(u_{0})\right)=0,
\label{rec}
\end{eqnarray}
for real Lagrange multipliers $\lambda_{0}^{j}, j=1,2$, will take the form
\begin{eqnarray}
u_{0}''+|u_{0}|^{2\sigma}u_{0}-\lambda_{0}^{1}u_{0}-i\lambda_{0}^{2}\partial_{x}u_{0}=0,\label{RE}
\end{eqnarray}
which can be written as a real system
\begin{eqnarray}
p_{0}''-\lambda_{0}^{1}p_{0}+\lambda_{0}^{2}q_{0}'+(p_{0}^{2}+q_{0}^{2})^{\sigma}q_{0}&=&0,\nonumber\\
q_{0}''-\lambda_{0}^{1}q_{0}-\lambda_{0}^{2}p_{0}'+(p_{0}^{2}+q_{0}^{2})^{\sigma}p_{0}&=&0.\label{fnls22_2}
\end{eqnarray}
Considering the polar representation $u_{0}(x)=e^{i\theta(x)}\rho(x)$ with real $\rho$ and $\theta$, and the orbit through $u_{0}$ by the symmetry group (\ref{symg})
$$\varphi(x,x_{0},\theta_{0})=G_{(\theta_{0},x_{0})}(u_{0})=\rho(x-x_{0})e^{i\theta(x-x_{0})+i\theta_{0}},$$ then
a two-parameter family of solitary-wave solutions of (\ref{nls}) can be obtained from the action of a one-parameter symmetry group determined by the Lagrange multipliers as
\begin{eqnarray}
\psi(x,t,\lambda_{0}^{1},\lambda_{0}^{2},x_{0},\theta_{0})&=&G_{(t\lambda_{0}^{1},t\lambda_{0}^{2})}(\varphi)\label{sw}\\
&=&\rho(x-t\lambda_{0}^{2}-x_{0})e^{i(\theta(x-t\lambda_{0}^{2}-x_{0})+\theta_{0}+\lambda_{0}^{1}t)}.\nonumber
\end{eqnarray}
The resolution of (\ref{RE}) leads to the explicit form, \cite{D}
\begin{eqnarray}
\rho(x)&=&(a(\sigma +1))^{1/2\sigma}({\rm sech}\sigma\sqrt{a}x)^{1/\sigma},\quad a=\lambda_{0}^{1}-\frac{(\lambda_{0}^{2})^{2}}{4},\nonumber\\
\theta(x)&=&\frac{\lambda_{0}^{2}}{2}x.\label{sw2}
\end{eqnarray}
In (\ref{sw}), the role of speed of the solitary wave is played by the second Lagrange multiplier $\lambda_{0}^{2}$, while the amplitude is determined by the combination of the multipliers given by the parameter $a$ in (\ref{sw2}).

For the numerical approximation of (\ref{nls2}), in particular when studying the dynamics of the solitary waves (\ref{sw}) and other localized solutions, a first typical step consists of approximating the ivp of (\ref{nls2}) by a periodic ivp on a long enough interval $\Omega=(l_i,l_s)$. More precisely, following \cite{BChen,Pasciak1982}, we assume:
\begin{itemize}
\item[(H1)] $\Omega$ is wide enough (it contains mainly the support of the solution of the ivp till time $T$) so that the difference in the supremum norm between the solution
of the ivp and the solution of the periodic ivp is negligible.
\end{itemize}
In a second approach, within the framework of the method of lines, the resulting periodic ivp is discretized in space. The choice considered here is given by a
Fourier pseudospectral method based on a uniform grid $x_{j}=l_i+j\Delta x, (j=0,\ldots,N-1, \Delta x=l/N$ of collocation points). We define
\begin{eqnarray*}
P_{N}(t)=(P_{N,0}(t),\ldots,P_{N,N-1}(t))^{T},\; Q_{N}(t)=(Q_{N,0}(t),\ldots,Q_{N,N-1}(t))^{T},
\end{eqnarray*}
where $(P_{N,j}(t),Q_{N,j}(t))$ approximates  $(p(x_j,t),q(x_j,t))$ in the sense that the vectors $P_{N}, Q_{N}$ satisfy the semidiscrete system
\begin{eqnarray}
\dot{P}_{N}&=&f(P_{N},Q_{N})=- A_N Q_{N}-(P_{N}.^2+Q_{N}.^2)^{\sigma}.Q_{N}, \nonumber \\
\dot{Q}_{N}&=&g(P_{N},Q_{N})=A_N P_{N}+(P_{N}.^2+Q_{N}.^2)^{\sigma}.P_{N}, \label{nlsdis}
\end{eqnarray}
where $A_N$ is a symmetric matrix which corresponds to the discretization of the second spatial derivative and $\cdot$ denotes pointwise product.

As stated in \cite{C}, the discretized system (\ref{nlsdis}) admits a canonical Hamiltonian structure with Hamiltonian function
\begin{eqnarray}
H_N(P_N,Q_N)=-\frac{l}{2}\bigg[ Q_N^T A_N Q_N+P_N^T A_N P_N+\frac{1}{\sigma+1}[(P_N.^2+Q_N.^2).^{\sigma+1}]^T {\bf 1}_N\bigg],
\label{HNnls}
\end{eqnarray}
where ${\bf 1}_N$ denotes the $N$-dimensional vector with all components equal to one.
Moreover,
$$I_{1,N}=\frac{l}{2}[P_N^T P_N+Q_N^T Q_N],$$
is an invariant quantity of this discretized system and
$$
I_{2,N}=\frac{l}{2}[P_N^T B_N Q_N-Q_N^T B_N P_N],$$
is a quasi-invariant of the system, where $B_N$ is the real antisymetric matrix which discretizes the first derivative in space. Finally,
\begin{eqnarray}
\frac{l}{N}I_{1,N}, \quad \frac{l}{N}I_{2,N},\quad \frac{l}{N}H_N,
\label{Ninv}
\end{eqnarray}
are the natural discretizations of the invariants (\ref{nls3a})-(\ref{nls3c}).
%

In addition to (H1), the following conditions are assumed:

\begin{itemize}
\item[(H2)]
$N$ is as large enough so that the difference between the solitary wave nodal values and the solution of (\ref{nlsdis}) with the corresponding initial condition can be considered negligible for $0\leq t\leq T$ in the supremum norm.
\item[(H3)] When integrating (\ref{nlsdis}) with a PLMM, the residual terms in the asymptotic expansion (\ref{thI}) of the error in the invariants (\ref{Ninv}) have a Landau constant which is independent of $N$ for $0\le t \le T$.
\end{itemize}

\subsubsection{Error in the Hamiltonian}{When $I=l H_{N}/N$, with $H_N$ the Hamiltonian (\ref{HNnls}) of the system (\ref{nlsdis}), and the method is symmetric, then (\ref{ham_smooth}) applies. This simplifies the computation of the integral in (\ref{thI}) multiplied by $l/N$ through the evaluation of the terms
\begin{eqnarray}
\frac{l}{N} P_N^{(l)}(t)^T Q_N^{(2k+1-l)}(t),\quad \frac{l}{N} Q_N^{(l)}(t)^T P_N^{(2k+1-l)}(t),
\label{pqder}
\end{eqnarray}
and
\begin{eqnarray}
\frac{l}{N} \int_{t_0}^t P_N^{(k+1)}(s)^T Q_N^{(k+1)}(s)ds.
\label{intpq}
\end{eqnarray}
Note that, due to hypotheses (H1)-(H3), we can evaluate these expressions  at the solitary wave solution $(p(t),q(t))$ in (\ref{sw}), (\ref{sw2}), instead of the corresponding semidiscrete solution $(P_{N},Q_{N})$ of (\ref{nlsdis}), and due to its regularity and the boundedness
\begin{eqnarray}
\frac{d^j}{d x^j}[\mbox{sech}(x)]^{\frac{1}{\sigma}} \le C_j e^{-\frac{|x|}{\sigma}},
\label{bounddsech}
\end{eqnarray}
for some constant $C_j$, then the terms in (\ref{pqder}) are bounded by
$$C_{2k+1} \int_{-\infty}^{\infty}e ^{-2 \sqrt{a} |y|} dy=2 C_{2k+1} \int_0^\infty e^{-2 \sqrt{a} y}dy=\frac{C_{2k+1}}{\sqrt{a}},$$
where the constants $C_{2k+1}$ may grow with $k$ but are independent of time.
As for (\ref{intpq}),
we observe that, considering (\ref{sw})-(\ref{sw2}), it is in fact a good approximation of
\begin{eqnarray}
\int_{t_0}^{t_n} \int_{-\infty}^{\infty} \frac{d^{k+1}}{ds^{k+1}} p(x,s)
\frac{d^{k+1}}{ds^{k+1}}q(x,s) dx ds, \label{boundnls}
\end{eqnarray}
where
\begin{eqnarray}
p(x,t)&=&\rho(x-t \lambda_0^2-x_0)\cos(\alpha(x,t)), \quad \alpha(x,t)=\frac{\lambda_0^2}{2}(x-t\lambda_0^2-x_0)+\theta_0+\lambda_0^1 t, \nonumber \\
q(x,t)&=&\rho(x-t \lambda_0^2-x_0)\sin(\alpha(x,t)).  \label{pq}
\end{eqnarray}
We then notice that, by Leibniz rule,
\begin{eqnarray}
\lefteqn{\frac{d^{k+1}}{dt^{k+1}} p(x,t)=\sum_{\tiny \begin{array}{c} j=0 \\ j=2r \end{array}}^{k+1} {\small \left( \begin{array}{c} k+1 \\ j \end{array} \right)} (-\lambda_0^2)^{k+1-j} \rho^{(k+1-j)}(x-t \lambda_0^2-x_0) a^j (-1)^r \cos(\alpha(x,t))} \nonumber \\
&&+\sum_{\tiny\begin{array}{c} j=0 \\ j=2r+1 \end{array}}^{k+1} {\small\left( \begin{array}{c} k+1 \\ j \end{array} \right)} (-\lambda_0^2)^{k+1-j} \rho^{(k+1-j)}(x-t \lambda_0^2-x_0)a^j (-1)^{r+1} \sin(\alpha(x,t)), \nonumber \\
\lefteqn{\frac{d^{k+1}}{dt^{k+1}} q(x,t)=\sum_{\tiny \begin{array}{c}j=0 \\ j=2r \end{array}}^{k+1} {\small \left( \begin{array}{c} k+1 \\ j \end{array} \right)} (-\lambda_0^2)^{k+1-j} \rho^{(k+1-j)}(x-t \lambda_0^2-x_0)a^j (-1)^r \sin(\alpha(x,t))} \nonumber \\
&&+\sum_{\tiny\begin{array}{c} j=0 \\ j=2r+1 \end{array}}^{k+1} {\small \left( \begin{array}{c} k+1 \\ j \end{array} \right) } (-\lambda_0^2)^{k+1-j} \rho^{(k+1-j)}(x-t \lambda_0^2-x_0)a^j (-1)^{r} \cos(\alpha(x,t)), \nonumber
\end{eqnarray}
from which it holds that
\begin{eqnarray}
\lefteqn{\frac{d^{k+1}}{dt^{k+1}} p(x,t) \frac{d^{k+1}}{dt^{k+1}} q(x,t)
=\rho_{k+1,sc}(x-t\lambda_0^2-x_0) \sin(\alpha(x,t)) \cos(\alpha(x,t))} \nonumber \\
&&+\rho_{k+1,cc}(x-t \lambda_0^2-x_0) \cos^2(\alpha(x,t))+\rho_{k+1,ss}(x-t \lambda_0^2-x_0) \sin^2(\alpha(x,t)),\nonumber
\end{eqnarray}
where it happens that $\rho_{k+1,cc}=-\rho_{k+1,ss}$. Thus,
\begin{eqnarray}
\lefteqn{\frac{d^{k+1}}{dt^{k+1}} p(x,t) \frac{d^{k+1}}{dt^{k+1}} q(x,t)} \label{trig} \\
&&=\frac{1}{2}\rho_{k+1,sc}(x-t \lambda_0^2-x_0 ) \sin(2 \alpha(x,t)) +\rho_{k+1,cc}(x-t \lambda_0^2-x_0) \cos(2 \alpha(x,t)).\nonumber
\end{eqnarray}
Then, using (\ref{trig}), the change of variables $y=x-s \lambda_0^2-x_0$, and Fubini's theorem, (\ref{boundnls}) can also be written as
\begin{eqnarray}
\lefteqn{\frac{1}{2}\int_{-\infty}^{\infty} \rho_{k+1,sc}(|y|) [\int_{t_0}^{t_n}\sin(\lambda_0^2 y+2 \theta_0+2 \lambda_0^1 s) ds]dy} \nonumber \\
 &&+ \int_{-\infty}^{\infty} \rho_{k+1,cc}(|y|)[\int_{t_0}^{t_n}\cos(\lambda_0^2 y+2 \theta_0+2 \lambda_0^1 s)ds]dy \nonumber \\
&=&-\frac{1}{4 \lambda_0^1}\int_{-\infty}^{\infty} \rho_{k+1,sc}(|y|)[\cos(\lambda_0^2 y+2 \theta_0+2 \lambda_0^1 t_n)-\cos(\lambda_0^2 y+2 \theta_0+2 \lambda_0^1 t_0)]dy \nonumber \\
&&+\frac{1}{2 \lambda_0^1}\int_{-\infty}^{\infty}  \rho_{k+1,cc}(|y|)[\sin(\lambda_0^2 y+2 \theta_0+2 \lambda_0^1 t_n)-\sin(\lambda_0^2 y+2 \theta_0+2 \lambda_0^1 t_0)]dy. \nonumber
\end{eqnarray}
Now, as $\rho_{k+1,sc}(y)$ and $\rho_{k+1,cc}(y)$  are both bounded by $D_{k+1}e^{-2\sqrt{a}|y|}$ for some constant $D_{k+1}$,
(\ref{boundnls}) can be bounded in modulus by
$$
\frac{3}{|\lambda_0^1|} D_{k+1}\int_0^{\infty}  e^{-2\sqrt{a}y}dy= \frac{3}{2|\lambda_0^1|\sqrt{a}} D_{k+1},$$
which is independent of time.

On the other hand, the terms corresponding to the invariant $I=lH_N/N$ in (\ref{inc}) can be written as
\begin{eqnarray}
\frac{l}{N}[\dot{Q}_N(t_n)^T E_{j,i,pp}(t_n)-\dot{P}_N(t_n)^T E_{j,i,qp}(t_n)], \nonumber \\
\frac{l}{N}[\dot{Q}_N(t_n)^T E_{j,i,pq}(t_n)-\dot{P}_N(t_n)^T E_{j,i,qq}(t_n)]. \label{ejidd}
\end{eqnarray}
In order to bound $E_{j,i,pp}, E_{j,i,qp}, E_{j,i,pq}, E_{j,i,qq}$,  we must look at $f_P((P_N(t),Q_N(t))^T)$, $f_Q((P_N(t),Q_N(t))^T)$, $g_P((P_N(t),Q_N(t))^T)$, $g_Q((P_N(t),Q_N(t))^T)$ and their time derivatives, as well as to the transition matrices associated to (\ref{tm}). We thus notice that
\begin{eqnarray}
f_P(P_N,Q_N)&=&-2 \sigma \mbox{diag}((P_N.^2+Q_N.^2)^{\sigma-1}.P_N.Q_N), \nonumber \\
f_Q(P_N,Q_N)&=&-A_N- \mbox{diag}((P_N.^2+Q_N.^2)^{\sigma-1}(P_N.^2+(2 \sigma+1) Q_N.^2)), \nonumber \\
g_P(P_N,Q_N)&=&A_N+ \mbox{diag}((P_N.^2+Q_N.^2)^{\sigma-1}((2 \sigma+1) P_N.^2+Q_N.^2)), \nonumber \\
g_Q(P_N,Q_N)&=&2 \sigma \mbox{diag}((P_N.^2+Q_N.^2)^{\sigma-1}P_N.Q_N). \label{jac}
\end{eqnarray}
We firstly see that $f_P(P_N,Q_N)$ and $g_Q(P_N,Q_N)$ are clearly bounded with time when $P_N$ and $Q_N$ are mainly the nodal projection of (\ref{pq}). On the other hand, $(P_{N}.^2+Q_{N}.^2)^{\sigma-1}.P_{N}(t).Q_{N}(t)$ is in fact approximating the values at the space grid of
$$
\frac{1}{2}\rho^{2 \sigma}(x-t \lambda_0^2-x_0)\sin(2 \alpha(x,t)).
$$
Because of this and (H1)-(H3), the transition matrices associated to (\ref{tm}) are diagonal matrices which elements at every grid value $x_j$ are mainly given by
\begin{eqnarray}
e^{- \lambda_{p,i,p} \sigma \int_{t_0}^{t_n} \rho^{2 \sigma}(x_j-t \lambda_0^2-x_0)\sin(2 \alpha(x_j,t))dt}, \quad
e^{\lambda_{q,i,q} \sigma \int_{t_0}^{t_n} \rho^{2 \sigma}(x_j-t \lambda_0^2-x_0)\cos(2 \alpha(x_j,t))dt}. \nonumber
\end{eqnarray}

We can thus see that the integrals in the exponents of these expressions are always bounded in modulus by
$$4 a (\sigma+1) \int_{t_0}^{t_n} e^{-2\sigma \sqrt{a}|x_j-t\lambda_0^2-x_0|}dt\le 8 a (\sigma+1) \int_0^{\infty} e^{-2\sigma \sqrt{a}u}du=4 \sqrt{a} (1+\frac{1}{\sigma}).$$
Therefore, the transition matrices are bounded independently of time $t_n$. On the other hand, we also notice that $g_P(P(t),Q(t))$ and $f_Q(P(t),Q(t))$ will be bounded with time.
It is true that the norm of $A_N$ can be quite large but, in any case, if the starting values correspond to functions which are regular enough in space, then the application of $A_{N}$ over the coefficients $E_{j,i,pp}(t)$ and $E_{j,i,qq}(t)$ leads to uniformly bounded in N terms. Finally, the time derivatives of all the functions in (\ref{jac}) will be clearly bounded in time. On the other hand, as the derivatives of (\ref{pq}) are also bounded independently of time and their modulus is integrable in $(-\infty,\infty)$, it happens that (\ref{ejidd}) can be bounded independently of time with bounds of the order of $\|E_{j,i,pp}(t_n)\|_{\infty}$, $\|E_{j,i,pq}(t_n)\|_{\infty}$, $\|E_{j,i,qp}(t_n)\|_{\infty}$, $\|E_{j,i,qq}(t_n)\|_{\infty}$.}

%


\subsubsection{Error in $I_1$}{As for the error in $lI_{1,N}/N$, looking again at (\ref{thI}),  the following terms are required to be bounded:
\begin{eqnarray}
&&\frac{l}{N}[P_N(t_n)^T E_{j,i,pp}(t_n)+Q_N(t_n)^T E_{ji,qp}(t_n)], \nonumber \\
&&\frac{l}{N}[P_N(t_n)^T E_{j,i,pq}(t_n)+Q_N(t_n)^T E_{ji,qq}(t_n)], \nonumber \\
&&\frac{l}{N}\int_{t_0}^{t_n}[ c_{j,p}P_N(s)^T P_N^{(j+1)}(s)+c_{j,q} Q_N(s)^T Q_N^{(j+1)}(s)]ds. \label{bi1}
\end{eqnarray}
With similar arguments to those of (\ref{ejidd}), we notice that the first two terms in these expressions are bounded in time. As for the last term, when $(\rho_p, \sigma_p)$ and $(\rho_q, \sigma_q)$ are symmetric, it vanishes for odd $j$ because, in such a case, $c_{j,p}=c_{j,q}=0$ for odd $j$. When $j=2k$, since
\begin{eqnarray}
P_N(s)^T P_N^{(2k+1)}(s)&=&\frac{d}{ds}[P_N(s)^T P_N^{(2k)}(s)]-\frac{d}{ds}[\dot{P}_N(s)^T P_N^{(2k-1)}(s)]
\nonumber \\
&&+\frac{d}{ds}[\ddot{P}_N(s)^T P^{(2k-2)}_N(s)]+\dots+\frac{1}{2}\frac{d}{ds}[P_N^{(k)}(s)^T P_N^{(k)}(s)],\nonumber \\
Q_N(s)^T Q_N^{(2k+1)}(s)&=&\frac{d}{ds}[Q_N(s)^T Q_N^{(2k)}(s)]-\frac{d}{ds}[\dot{Q}_N(s)^T Q_N^{(2k-1)}(s)] \nonumber \\
&&
+\frac{d}{ds}[\ddot{Q}_N(s)^T Q_N^{(2k-2)}(s)]+\dots+\frac{1}{2}\frac{d}{ds}[Q_N^{(k)}(s)^T Q_N^{(k)}(s)], \nonumber
\end{eqnarray}
the last line in (\ref{bi1}) reduces to
\begin{eqnarray}
\lefteqn{\hspace{-1cm}\frac{l}{N}\bigg[ c_{2k,p} \big[ [P_N(t_n)^T P_N^{(2k)}(t_n)-P_N(t_0)^T P_N^{(2k)}(t_0)]} \nonumber \\
&&
-[\dot{P}_N(t_n)^T P_N^{(2k-1)}(t_n)-\dot{P}_N(t_0)^T P_N^{(2k-1)}(t_0)] \nonumber \\
&&
+[\ddot{P}_N(t_n)^T P_N^{(2k-2)}(t_n)-\ddot{P}_N(t_0)^T P_N^{(2k-2)}(t_0)]+\dots \nonumber \\
&&
+\frac{1}{2}[P_N^{(k)}(t_n)^T P_N^{(k)}(t_n)-P_N^{(k)}(t_0)^T P_N^{(k)}(t_0)]\big]\nonumber \\
&&
\hspace{-0.7cm}
+c_{2k,q} \big[ [Q_N(t_n)^T Q_N^{(2k)}(t_n)-Q_N(t_0)^T Q_N^{(2k)}(t_0)] \nonumber \\
&&
-[\dot{Q}_N(t_n)^T Q_N^{(2k-1)}(t_n)-\dot{Q}_N(t_0)^T Q_N^{(2k-1)}(t_0)] \nonumber \\
&&
+[\ddot{Q}_N(t_n)^T Q_N^{(2k-2)}(t_n)-\ddot{Q}_N(t_0)^T Q_N^{(2k-2)}(t_0)]+\dots \nonumber \\
&&
+\frac{1}{2}[Q_N^{(k)}(t_n)^T Q_N^{(k)}(t_n)-Q_N^{(k)}(t_0)^T Q_N^{(k)}(t_0)]\big]\bigg].\label{simI1}
\end{eqnarray}
We note now that all terms in the previous expressions are in fact bounded by
$$C_{2k}' \int_{-\infty}^{\infty} \rho_{2k}(|x-t_n \lambda_0^2-x_0|)dx \le 2 C_{2k}' \int_{t_n\lambda_0^2+x_0}^{\infty} e^{-2\sqrt{a}|x-t_n\lambda_0^2-x_0|}dx \le \frac{C_{2k}'}{\sqrt{a}},$$
or
$$C_{2k}' \int_{-\infty}^{\infty} \rho_{2k}(|x-t_0 \lambda_0^2-x_0|)dx \le 2 C_{2k}' \int_{t_0\lambda_0^2+x_0}^{\infty} e^{-2\sqrt{a}|x-t_0\lambda_0^2-x_0|}dx \le \frac{C_{2k}'}{\sqrt{a}},$$
for a certain constant $C_{2k}'$ and where $\rho_{2k}(y)$ is a certain function which comes from the successive $2k$ derivatives of $\mbox{sech}(\sigma \sqrt{a}y)^{\frac{1}{\sigma}}$ (see (\ref{bounddsech})).
From this, it is finally proved that all terms in (\ref{bi1}) are bounded independently of time.}

\subsubsection{Error in $I_2$}{
For the error in $lI_{2,N}/N$, similar arguments to those in $lH_N/N$ could be performed, but it is in fact not necessary since, at the relative equilibrium solution, $\nabla I_{2,N}$ will be, except for negligible errors, a linear combination of $\nabla I_{1,N}$ and $\nabla H_N$, cf. (\ref{rec}), \cite{D}.}

\subsubsection{Concluding results}{

From all the above, the following conclusions hold.

\begin{theorem}
\label{th1}
Consider the semidiscrete system (\ref{nlsdis}), with $N$ and $(l_i,l_s)$ satisfying (H1)-(H3) for $u(x,t)$ a solitary wave solution of the form (\ref{sw}), (\ref{sw2}). Assume that (\ref{nlsdis}) is integrated in time by some $r$th-order, symmetric PLMM (\ref{eq12}) where $\rho_{p}, \rho_{q}$ have no common unitary roots except $x_{1}=1$, with time step $\Delta t$. Then:
\begin{enumerate}
\item[(i)] If the starting values differ from the exact ones in $O(\Delta t^r)$, the errors in $lI_{1,N}/N$,  $lI_{2,N}/N$ and $lH_N/N$  are $O(\Delta t^r)$ with a Landau constant independent of time for $t-t_0=O(\Delta t^{-1})$.
\item[(ii)] If the starting values differ from the exact ones in $O(\Delta t^{r+1})$, the errors in $lI_{1,N}/N$, $lI_{2,N}/N$ and $lH_N/N$  are $O(\Delta t^r)$ with a Landau constant independent of time for $t-t_0=O(\Delta t^{-2})$.
\end{enumerate}
\end{theorem}

\begin{remark}
We notice that, in the non-symmetric case, the simplification (\ref{ham_smooth}) for the error in the Hamiltonian cannot be done. Then, associated to the odd indexes $j$ in (\ref{secline}), using integration-by-parts, there would appear terms of the form
$$
\int_{t_0}^{t_n} P_N(t) Q_N^{(2k+3)}(t)dt,$$
which would give rise to terms in $\cos()\sin()$, which integral is bounded with arguments similar to those above, and others in $\cos()^2$ which integral grows linearly with time. The latter would explain the linear error growth with time with these methods.

Similarly, for the error in $I_{1}$, for odd $j$ in (\ref{bi1}), (\ref{simI1}) cannot be applied and the integrals of functions containing a factor $\sin()^2$ or $\cos()^2$ are likely to grow linearly.
\label{rem1}
\end{remark}}


\subsection{Approximation to solitary wave solutions of the Boussinesq Bona-Smith systems}
\label{sec22}
This section is devoted to the corresponding analysis of the second case study, the three-parameter Boussinesq system (\ref{bb}). Among its relevant properties, we are here interested in two of them. One is the existence of the following conserved quantities for smooth, localized solutions of the ivp or smooth solutions of the periodic ivp of (\ref{bb}). The first are the linear quantities
\begin{eqnarray}
M_1(\eta,w)= \int_{\Omega} \eta dx, \quad M_2(\eta,w)= \int_{\Omega} w dx,\label{invbb1}
\end{eqnarray}
which hold in the general case, while
\begin{eqnarray}
I(\eta,w)&=&\int_{\Omega} (\eta w+ b \eta_x w_x) dx, \label{invbb2}\\
H(\eta,w)&=&\frac{1}{2}\int_{\Omega} [\eta^2+w^2-a w_x^2 -c \eta_x^2 +\eta w^2]dx,\label{invbb3}
\end{eqnarray}
are conserved only in the case $b=d$. In (\ref{invbb1})-(\ref{invbb3}),  $\Omega=\mathbb{R}$ or an interval of periodicity $(l_i,l_s), l=l_s-l_i>0$. We note that (\ref{invbb3}) is the Hamiltonian function of a noncanonical Hamiltonian structure of (\ref{bb}), \cite{BCS2}.

On the other hand, some of the systems (\ref{bb}) admit solitary wave solutions
\begin{eqnarray}
\eta(x,t)=\eta_{s}(x-c_{s}t-x_{0}), \quad w(x,t)=w_{s}(x-c_{s}t-x_{0}), \label{simh}
\end{eqnarray}
for $c_{s}\neq 0, x_{0}\in\mathbb{R}$ and some functions $\eta_{s}$ and $w_{s}$ satisfying
\begin{eqnarray}
-c_{s}\eta_{s} +w_{s}+w_{s}\eta_{s} +aw_{s}''-c_{s}b\eta_{s}''&=&0,\nonumber\\
-c_{s}w_{s}+\eta_s+\frac{w_{s}^{2}}{2}+c\eta_s''-c_{s}d w_{s}''&=&0,\label{RE2}
\end{eqnarray}
such that
\begin{eqnarray}
|\eta_{s}(x)| \le C e^{-\alpha|x|}, \quad |w_{s}(x)| \le C e^{-\alpha|x|},\quad x\in\mathbb{R},\label{asymp}
\end{eqnarray}
for some $C,\alpha>0.$ The existence is limited by the values of the parameters $a,b,c,d$ and speed $c_{s}$, cf. e.~g. \cite{DougalisM2008} and references therein.

Note that in the Hamiltonian case $b=d$, (\ref{RE2}) can be written as the constrained critical point condition
\begin{eqnarray}
\delta (H(\eta_{s},w_{s})-c_{s}I(\eta_{s},w_{s}))=0,\label{RE2b}
\end{eqnarray}
with the speed $c_{s}$ as a Lagrange multiplier. The invariant $I$ determines the infinitesimal generator of the symmetry group of (\ref{bb}) of spatial translations
\begin{eqnarray*}
G_{\epsilon}(\eta(x),w(x))=(\eta(x-\epsilon),w(x-\epsilon)),\quad \epsilon, x\in\mathbb{R},
\end{eqnarray*}
and therefore the solitary waves (\ref{simh}) can be obtained from the application of the one-parameter symmetry group $G_{tc_{s}}, t\geq 0$, to the orbit of the profile $(\eta_{s},w_{s})$, solution of (\ref{RE2}). Explicit formulas for the solitary waves (\ref{simh}) are only known for specific values of the speed, namely
\begin{eqnarray}
&&\eta_{s}(x)=\frac{3(1-\beta^{2})}{\beta^{2}}{\rm sech}^{2}(\lambda x),\quad w_{s}(x)=\beta\eta_{s}(x),\label{sw3}\\
&&\lambda=\frac{1}{2}\sqrt{\frac{2(1-\beta^{2})}{(a-b)\beta^{2}+2b}},\quad c_{s}=\frac{2-\beta^{2}}{\beta},\nonumber
\end{eqnarray}
for the parameter $\beta\neq 0$ satisfying suitable conditions, \cite{MChen}.

The spatial discretization of the periodic ivp of (\ref{bb}) on a long enough interval $\Omega$ with a Fourier collocation method leads to the semidiscrete system
\begin{eqnarray}
\dot{\Gamma}_{N}&=& -(I_N-b B_N^2)^{-1}\left( B_N (I_N+a B_N^2)W_{N}+[B_N W_{N}. \Gamma_{N}+ W_{N} . B_N \Gamma_{N}]\right),\nonumber \\
\dot{W}_{N}&=& -(I_N-d B_N^2)^{-1}\left(B_N (I+c B_N^2)\Gamma_{N}+[W_{N}.B_N W_{N}]\right), \label{sdbb}
\end{eqnarray}
where, if $(\eta,w)$ is the corresponding solution of the periodic ivp, then
\begin{eqnarray*}
&&\Gamma_{N}(t)=(\Gamma_{N,0}(t),\ldots,\Gamma_{N,N-1}(t))^{T},\; \Gamma_{N,j}(t)\approx \eta(x_{j},t),\\
&&W_{N}(t)=(W_{N,0}(t),\ldots,W_{N,N-1}(t))^{T},\; W_{N,j}(t)\approx w(x_{j},t), \quad j=0,\dots,N-1,
\end{eqnarray*}
and $B_N$ is the antisymmetric matrix which discretizes the first derivative in space \cite{C}.

%

\begin{theorem}
System (\ref{sdbb}) has the following invariant quantities:
\begin{eqnarray*}
 M_{1,N}= \Gamma_N^T {\mathbf 1}_N, \quad M_{2,N}= W_N^T {\mathbf 1}_N,
\end{eqnarray*}
and, in addition, when $b=d$,
\begin{eqnarray}
I_N( \Gamma_N, W_N)&=&\Gamma_N^T W_N+b \Gamma_N^T B_N^T B_N W_N, \nonumber \\
H_N(\Gamma_N, W_N)&=&\frac{1}{2}[ \Gamma_N^T \Gamma_N+W_N^T W_N-a  W_N^T B_N^T B_N W_N-c  \Gamma_N^T B_N^T B_N \Gamma_N+\Gamma_N^T W_N.^2 ].\nonumber
\end{eqnarray}
\end{theorem}
\begin{proof}
We first notice that
\begin{eqnarray}
\nabla H_N(\Gamma_N, W_N)=\left( \begin{array}{c} \Gamma_N+c B_N^2 \Gamma_N+\frac{1}{2} W_N.^2 \\ W_N+a B_N^2 W_N+W_N \cdot \Gamma_N. \end{array} \right),
\label{gradH}
\end{eqnarray}
and it happens that (\ref{sdbb}) can be written as
\begin{eqnarray}
\left( \begin{array}{c} \dot{\Gamma}_N \\ \dot{W}_N \end{array} \right)&=&\left( \begin{array}{c} f (\Gamma_N, W_N) \\ g(\Gamma_N, W_N)  \end{array} \right) \label{sys}\\
&:=&\left( \begin{array}{cc} (I_N-b B_N^2)^{-1} B_N & 0 \\ 0 & -(I_N-d B_N^2)^{-1} B_N \end{array} \right) J_N \nabla H_N( \Gamma_N, W_N),\nonumber
\end{eqnarray}
where
$$J_N=\left( \begin{array}{cc} 0 & I_N \\ -I_N & 0 \end{array} \right).$$
Then
\begin{eqnarray}
\nabla M_{1,N}^T \left( \begin{array}{c} f (\Gamma_N, W_N) \\ g(\Gamma_N, W_N)  \end{array} \right)&=& ( {\mathbf 1}_N \, \, 0)^T \left( \begin{array}{c} f (\Gamma_N, W_N) \\ g(\Gamma_N, W_N)  \end{array} \right)=0, \nonumber \\
\nabla M_{2,N}^T \left( \begin{array}{c} f (\Gamma_N, W_N) \\ g(\Gamma_N, W_N)  \end{array} \right)&=& ( 0 \, \, {\mathbf 1}_N)^T \left( \begin{array}{c} f (\Gamma_N, W_N) \\ g(\Gamma_N, W_N)  \end{array} \right)=0, \nonumber
\end{eqnarray}
where the last equalities come from (\ref{sys}) and the identity $B_N {\mathbf 1}={\mathbf 0}$. (This can be directly deduced from the fact that the derivative of a constant function,  which is itself a trigonometric function, vanishes, or directly from the expression of $B_N$, see \cite{C}.)

On the other hand, when $b=d$, because of the antisymmetry of the product of the two matrices appearing in (\ref{sys}), the system (\ref{sdbb}) admits a (noncanonical) Hamiltonian formulation with Hamiltonian given by $H_{N}$. Consequently, $H_{N}$ is a conserved quantity of the semidiscrete system.
As for $I_N$, it suffices to notice that
$$\nabla I_N \left( \begin{array}{c} \Gamma_N \\ W_N \end{array} \right)= \left( \begin{array}{c} (I_N-b B_N^2) W_N \\ (I_N-b B_N^2) \Gamma_N \end{array} \right),$$
and thus, using (\ref{sys}),
\begin{eqnarray}
\lefteqn{\hspace{-1cm}\nabla I_N(\Gamma_N, W_N)^T \left( \begin{array}{c} f (\Gamma_N, W_N) \\ g(\Gamma_N, W_N)  \end{array} \right)= -W_N^T B_N (W_N \cdot \Gamma_N)- \Gamma_N^T (W_N \cdot B_N W_N)}
\nonumber \\
&&= -\sum_{i,j} W_{N,i} (B_N)_{i,j} W_{N,j} \Gamma_{N,j}- \sum_{j,i} \Gamma_{N,j} W_{N,j} (B_N)_{j,i} W_{N,i}=0,
\nonumber
\end{eqnarray}
where the first and last equality are due to the antisymmetry of $B_N$.
\end{proof}

\begin{remark}
We notice that  $l M_{1,N}/N $, $l M_{2,N}/N$, $l I_N/N $ and $l H_N/N $ are the natural discretizations of the invariants (\ref{invbb1})-(\ref{invbb3}).
\end{remark}

\begin{remark}
It is well-known that any standard LMM conserves the linear invariants of the system they integrate \cite{Gear}. Although we believe the result is not true for any general linear invariant when considering PLMMs, in the case of  the linear invariants $M_{1,N}$ and $M_{2,N}$, the same argument as in \cite{Gear} can be applied because they just depend on one of the variables, either $\Gamma_N$ or $W_N$. Then, any PLMM conserves those invariants.
\label{remlin}
\end{remark}

We will assume from now on that similar hypotheses to (H1)-(H3) in Subsection \ref{sec21} hold. On the other hand, because of Remark \ref{remlin} and due to the fact that, at the solitary wave, $\nabla H_N$ will be proportional to $\nabla I_N$ (c.f. (\ref{RE2b})), we will just center our analysis on the error in the Hamiltonian, knowing that similar results will hold for the invariant $I$.

\subsubsection{Error growth in the Hamiltonian when considering the smooth part of the numerical solution}
{Note that, since the Hamiltonian structure is not canonical, (\ref{secline}) cannot be applied and we must estimate the third line in (\ref{thI}) directly. More precisely,
\begin{eqnarray}
\hspace{-0.2cm}\frac{l}{N} \int_{t_0}^t (\nabla H_N)^T (\Gamma_N(u), W_N(u))\left( \begin{array}{c} c_{j,\Gamma} \Gamma_N^{(j+1)}(u) \\ c_{j,W} W_N^{(j+1)}(u) \end{array} \right)du, j=r,\dots,2r-1.
\label{smoothbb}
\end{eqnarray}
When both methods being used are symmetric, the only nonvanishing expressions above correspond to $j=2k$ ($k=\frac{r}{2},\dots,r-1$). Considering (\ref{gradH}) and (H1), (H2), for large enough $\Omega$ and $N$, these are in fact a good approximation of
\begin{eqnarray}
\lefteqn{\hspace{-1cm}c_{2k,\Gamma} \int_{t_0}^t \int_{-\infty}^{\infty} [\eta+c \eta_{xx}+\frac{1}{2} w^2] \frac{d^{2k+1}}{du^{2k+1}}\eta dx du} \nonumber \\
&& +c_{2k,W} \int_{t_0}^t \int_{-\infty}^{\infty} [w+ \eta w+ a w_{xx}] \frac{d^{2k+1}}{du^{2k+1}}w dx du.
\label{erhbb}
\end{eqnarray}
If we assume now that $\eta, w$ are of the form (\ref{simh}) with even $\eta_s$ and $w_s$,
it happens that, for a fixed value of time $u$, the expressions in the integrals in (\ref{erhbb}) are antisymmetric with respect to $c_su+x_0$. Because of that, the integral in space vanishes for every value of $u\in[t_0,t]$ and thus also the integral in time.
}

\subsubsection{Error growth in the Hamiltonian when considering the non-smooth part of the numerical solution}
{Now, we will see how to bound the terms in (\ref{inc}), that is, those associated to $E_{j,i,\Gamma \Gamma}$, $E_{j,i,\Gamma W}$, $E_{j,i,W \Gamma}$ and $E_{j,i,WW}$. To this end, considering that in the Hamiltonian case $b=d$,
\begin{eqnarray}
f_{\Gamma}(\Gamma_N,W_N)&=&-(I_N-d B_N^2)^{-1}[\mbox{diag}(B_N W_N)+\mbox{diag}(W_N) B_N], \nonumber \\
g_{W}(\Gamma,W)&=&-(I_N-d B_N^2)^{-1}[\mbox{diag}(B_N W_N)+\mbox{diag}(W_N) B_N], \label{fpgq}
\end{eqnarray}
and following (\ref{tm}), it is necessary to see if the transition matrices associated to the following problems are bounded
\begin{eqnarray}
\dot{E}&=&-\lambda_{\Gamma, i, \Gamma}(I-d B_N^2)^{-1}[\mbox{diag}(B_N W_N)+\mbox{diag}(W_N) B_N]E, \nonumber \\
\dot{E}&=&-\lambda_{W, i, W }(I-d B_N^2)^{-1}[\mbox{diag}(B_N W_N)+\mbox{diag}(W_N) B_N]E. \label{E}
\end{eqnarray}
As both differential systems are of the same type, we will simplify the notation for the scalar as $\lambda$. We will see that the conclusion being drawn is quite independent of that value. We first notice that $E$ is the vector which is approximating at every node of the space discretization the solution of
$$e_t= -\lambda (1-d \partial_x^2)^{-1} \partial_x (we),$$
where $\partial_x$ denotes differentiation with respect to the variable $x$. Using Fourier analysis, when $d>0$, see \cite{BCS1}, this can be written as
\begin{eqnarray}
e_t(x,t)=-\lambda \int_{-\infty}^{\infty} k_d(x-y)w(y,t)e(y,t)dy,
\label{et}
\end{eqnarray}
where
$$
k_d(z)=\frac{1}{2}\sqrt{d} \mbox{ sign}(z) e^{-\frac{|z|}{\sqrt{d}}}.
$$
We then notice that, under the assumption, cf. (\ref{asymp})
\begin{eqnarray}
|w(y,t)| \le C e^{-\alpha|y-c_st -x_0|}, \mbox{ for }\alpha>0,
\label{cotaw}
\end{eqnarray}
it holds that
\begin{eqnarray}
\int_{-\infty}^{\infty} |k_d(x-y) w(y,s)| dy &\le&  \frac{C \sqrt{d}}{2} \int_{-\infty}^{\infty} e^{-\frac{|v|}{\sqrt{d}}}e^{-\alpha |x-v-c_s s -x_0|}dv. \nonumber
\end{eqnarray}
We then distinguish between the following cases:
\begin{list}{$\bullet$}{}
\item
$x-c_s s-x_0<0$. In this case, if $\alpha \neq \frac{1}{\sqrt{d}}$, it is easily deduced that
\begin{eqnarray}
\lefteqn{\int_{-\infty}^{\infty} e^{-\frac{|v|}{\sqrt{d}}}e^{-\alpha |x-v-c_s s -x_0|}dv } \nonumber \\
&=&\frac{e^{\frac{1}{\sqrt{d}}(x-c_s s -x_0)}+e^{\alpha(x-c_s s-x_0)}}{\frac{1}{\sqrt{d}}+\alpha}+\frac{e^{\alpha(x-c_s s -x_0)}-e^{\frac{1}{\sqrt{d}}(x-c_s-x_0)}}{\frac{1}{\sqrt{d}}-\alpha}.
\label{for1}
\end{eqnarray}
If $\alpha=1/\sqrt{d}$, the second term in (\ref{for1}) will be  $e^{\alpha(x-c_s s-x_0)}(c_s s+x_0-x)$.
\item
$x-c_s s -x_0>0$. If $\alpha \neq \frac{1}{\sqrt{d}}$,
\begin{eqnarray}
\lefteqn{\int_{-\infty}^{\infty} e^{-\frac{|v|}{\sqrt{d}}}e^{-\alpha |x-v-c_s s -x_0|}dv} \nonumber \\
&=&\frac{e^{-\frac{1}{\sqrt{d}}(x-c_s s-x_0)}+e^{-\alpha(x-c_s s-x_0)}}{\frac{1}{\sqrt{d}}+\alpha}+\frac{e^{-\alpha(x-c_s s-x_0)}
-e^{-\frac{1}{\sqrt{d}}(x-c_s s -x_0)}}{\frac{1}{\sqrt{d}}-\alpha}.
\label{for2}
\end{eqnarray}
If $\alpha=1/\sqrt{d}$, the second term in (\ref{for2}) will be $e^{-\alpha(x-c_s s-x_0)}(x-c_s s -x_0)$.
\end{list}

Therefore, by using the integral calculus fundamental theorem in (\ref{et}), if $\alpha \neq 1/\sqrt{d}$, there exists a constant $C'>0$ (which just depends on $C,\lambda$, $\alpha$ and $\sqrt{d}$) such that, for fixed $x$ and $t$,
\begin{eqnarray}
|e(x,t)|\le |e(x,t_0)|+ C' \int_{t_0}^t  e^{-\min\{\alpha, \frac{1}{\sqrt{d}} \}|x-c_s s -x_0|} \|e(\cdot, s)\|_{\infty} ds.
\label{bound1}
\end{eqnarray}
It is clear that the integral in (\ref{bound1}), as continuous function of $x$, attains the maximum at some $x^*(t)\in [c_s t_0-x_0,c_s t-x_0]$. Then

\begin{eqnarray}
\|e(\cdot,t)\|_{\infty} \le \|e(\cdot,t_0)\|_{\infty}+  C' \int_{t_0}^t  e^{-\min\{\alpha, \frac{1}{\sqrt{d}} \}|x^*(t)-c_s s-x_0|} \|e(\cdot, s)\|_{\infty} ds,
\nonumber
\end{eqnarray}
which, by Gronwall's lemma, implies, for some constant $C'$, that
$$
\|e(\cdot,t)\|_{\infty} \le \|e(\cdot,t_0)\|_{\infty} \, e^{ C' \int_{t_0}^t e^{-\min\{\alpha, \frac{1}{\sqrt{d}} \}|x^*(t)-c_s s-x_0|} ds}.
$$
As, for $\beta>0$, it happens that
\begin{eqnarray}
\int_{t_0}^t e^{-\beta |x^*(t)-c_s s-x_0|} ds
\le \frac{2}{\beta c_s},
\nonumber
\end{eqnarray}
it follows that
$$
\|e(\cdot,t)\|_{\infty} \le \|e(\cdot,t_0)\|_{\infty} \, e^{\frac{2 C'}{\min\{\alpha, \frac{1}{\sqrt{d}} \}c_s}}.
$$
This explains that the transition matrices associated to both systems in (\ref{E}) are bounded with time, and a similar conclusion holds, by similar arguments, for the case $\alpha=1/\sqrt{d}$.

Finally, the bound for (\ref{fun_jac}) and their time derivatives requires enough regularity in time of (\ref{fpgq}) as well as that of
\begin{eqnarray}
f_W(\Gamma_N, W_N)&=&-(I-d B_N^2)^{-1}[B_N(I+a B_N^2)-\mbox{diag}(\Gamma_N)B_N-\mbox{diag}(B_N \Gamma_N)], \nonumber \\
g_{\Gamma}(\Gamma_N,W_N)&=& -(I-d B_N^2)^{-1}B_N(I+c B_N^2),
\nonumber
\end{eqnarray}
at the solution $(\Gamma,W)$.
We notice that, in fact, the expression for $g_{\Gamma}$ does not depend on $\Gamma_N$ either on $W_N$. Moreover, the regularity in time of $f_\Gamma$, $f_W$, $g_W$ correspond to the regularity of $\Gamma_N(t)$ and $W_N(t)$ itself.

Because of all this, following the remarks at the end of Section \ref{secprel}, the terms of the error associated to the non-common unitary roots behave as $O(\Delta t^r)$ for $t-t_0=O(\Delta t^{-1})$, and the same happens for the first line in (\ref{thI}) when $I=lH_N/N$ whenever $l\|\nabla H_N(\Gamma_N(t_n),W(t_n))\|_1/N=l\|(\dot{W}(t_n),-\dot{\Gamma}_N(t_n))\|_1/N$, which approximates
$$\int_{-\infty}^{\infty} |w_t(x,t)|dx+\int_{-\infty}^{\infty} |\eta_t(x,t)|dx,$$
is bounded with time. Note that, due to the special form of solitary waves (\ref{simh}), that happens whenever
\begin{eqnarray}
\int_{-\infty}^{\infty} |\eta_s'(x)|dx<\infty, \quad \int_{-\infty}^{\infty} |w_s'(x)|dx<\infty.
\label{acotdersbb}
\end{eqnarray}
(Obviously, the functions in (\ref{sw3}) satisfy this.)

On the other hand, if the starting procedure has order $r+1$, the coefficients $E_{r,i,\Gamma \Gamma}$, $E_{r,i,\Gamma W}$, $E_{r,i, W \Gamma}$ and $E_{r,i, W W}$ vanish. Therefore, if the method is symmetric, considering that (\ref{smoothbb}) vanishes, the error in the Hamiltonian behaves as $O(\Delta t^{r+1})$ for $t-t_0=O(\Delta t^{-1})$.

Moreover, if the starting procedure has order $r+2$, according to the end of Section \ref{secprel}, $E_{r+1,i,\Gamma W}\equiv E_{r+1,i,W \Gamma}\equiv 0$ and the initial conditions $E_{r+1,i,\Gamma \Gamma}(t_0)$ and $E_{r+1,i,W W}(t_0)$ satisfy system (\ref{icrm2}) where $p$ and $q$ should be substituted by $\Gamma_N$ and $W_N$. By assuming now that $\eta_s$ and $w_s$ in (\ref{simh}) are even, the $(r+1)$- time derivatives of $\eta$ and $w$ in (\ref{simh}) are odd functions in the space variable when $r$ is even. Because of that, if $D_N$ is the matrix which reverses the order of the components of a vector, it is natural that
$$
D_N \Gamma_N^{(r+1)}(t_0)=-\Gamma_N^{(r+1)}(t_0), \quad D_N W_N^{(r+1)}(t_0)=-W_N^{(r+1)}(t_0),$$
and so it will happen that
\begin{eqnarray}
D_N E_{r+1,i,\Gamma \Gamma}(t_0)=-E_{r+1,i,\Gamma \Gamma}(t_0), \quad D_N E_{r+1,i,WW}(t_0)=-E_{r+1,i,WW}(t_0).
\label{eci}
\end{eqnarray}
We then notice that $D_N B_N^2=B_N^2 D_N$, $D_N B_N D_N=B_N$ and, for any vector $X_N$,
$$ D_N \mbox{diag}(X_N)= \mbox{diag}(D_N X_N) D_N.$$
This implies, using also that $D_N W_N=W_N$ (because $w$ is an even function), that $D_N E_{r+1,i,\Gamma \Gamma}$ and $D_N E_{r+1,i,W W}$ satisfy the same differential systems than $E_{r+1,i,\Gamma \Gamma}$ and $E_{r+1,i,WW}$, i.e. (\ref{E}). Then, from the uniqueness of the solution of an initial value problem, using the relation (\ref{eci}) between the initial conditions, it holds that
$$
D_N E_{r+1,i,\Gamma \Gamma}(t)=-E_{r+1,i,\Gamma \Gamma}(t), \quad D_N E_{r+1,i,WW}(t)=-E_{r+1,i,WW}(t).
$$
Thus, considering also (\ref{gradH}) and using that $\Gamma_N$ and $W_N$ correspond to the nodal values of an even function, the expressions
$$
\frac{l}{N} \nabla H_N(\Gamma_N,W_N)^T E_{r+1,i,\Gamma \Gamma}, \quad \frac{l}{N} \nabla H_N(\Gamma_N,W_N)^T E_{r+1,i,WW}$$
approximate the integrals in space of odd functions, which can be considered negligible.

Because of all this, when the PLMM is symmetric (which implies even order) and the starting procedure has order $r+2$, the error in the Hamiltonian is $O(\Delta t^{r+2})$ till time $O(\Delta t^{-1})$.

}

\subsubsection{Concluding results}
{All the above leads to the following:

\begin{theorem}
\label{th2}
Consider the semidiscrete system (\ref{sdbb}) with $b=d>0$ and where $N$ and $\Omega$ satisfy analogous conditions to (H1)-(H3) for $(\eta(x,t),w(x,t))$ a solitary wave solution of the form (\ref{simh}) with even and regular enough functions $\eta_s$ and $w_s$ which satisfy (\ref{asymp}) and (\ref{acotdersbb}). Assume that (\ref{sdbb}) is integrated in time by some $r$th-order, symmetric PLMM (\ref{eq12}) where $\rho_{p}, \rho_{q}$ have no common unitary roots except $x_{1}=1$, with time step $\Delta t$. Then:
\begin{enumerate}
\item[(i)] If the starting values differ from the exact ones in $O(\Delta t^r)$, the errors in $lH_N/N$ and $lI_N/N$  are $O(\Delta t^r)$ with a Landau constant independent of time for $t-t_0=O(\Delta t^{-1})$.
\item[(ii)] If the starting values differ from the exact ones in $O(\Delta t^{r+1})$, the errors in $lH_N/N$ and $lI_N/N$ are $O(\Delta t^{r+1})$ with a Landau constant independent of time for $t-t_0=O(\Delta t^{-1})$.
\item[(iii)] If the starting values differ from the exact ones in $O(\Delta t^{r+2})$, the errors in $lH_N/N$ and $lI_N/N$ are $O(\Delta t^{r+2})$ with a Landau constant independent of time for $t-t_0=O(\Delta t^{-1})$.
\end{enumerate}
\end{theorem}

\begin{remark}
\label{remark32}
We notice that, in the non-symmetric case, terms with even $(j+1)$-derivatives in (\ref{smoothbb}) should also be considered. The corresponding integrands similar to (\ref{erhbb}) would then be even instead of odd and so the integral in space would necessarily not vanish. This explains that, with non-symmetric methods, the growth of the error in the invariants is linear from the beginning when $r$ is odd and is bounded at the beginning when $r$ is even, cf. Section \ref{sec3}.

\end{remark}}

%

\section{Numerical experiments}
\label{sec3}
In this section, a computational study of the use of PLMMs for the semidiscrete systems (\ref {nlsdis}) and (\ref{sdbb}) is performed. Part of the experiments will include comparisons among different methods and to this end we will consider three LMMs, as in \cite{CDR1}:
\begin{enumerate}
\item
The symmetric PLMM of second order (SPLMM2, \cite{CH})
\begin{eqnarray*}
\rho_p(x)&=&(x-1)(x+1), \quad \sigma_p(x)=2x, \nonumber \\
\rho_q(x)&=&(x-1)(x^2+1), \quad \sigma_q(x)=x^2+x. \label{plmm2}
\end{eqnarray*}
\item
The nonsymmetric PLMM where one of the methods is a symmetric LMM and the other one is the second-order Adams method (NSPLMM2)
\begin{eqnarray*}
\rho_p(x)&=&(x-1)(x+1), \quad \sigma_p(x)=2x, \nonumber \\
\rho_q(x)&=&x(x-1), \quad \sigma_q(x)=\frac{3}{2}x-\frac{1}{2}. \label{sim_nosim}
\end{eqnarray*}
\item
The nonsymmetric, nonpartitioned third-order Adams method (NSNPLMM3)
\begin{eqnarray*}
\rho_p(x)=\rho_q(x)=x^2(x-1), \; \sigma_p(x)=\sigma_q(x)=\frac{23}{12}x^2-\frac{16}{12}x+\frac{5}{12}. \label{adams3}
\end{eqnarray*}
\end{enumerate}
For simplicity, the experiments were made with intervals $\Omega=(-L,L)$ (so $l=2L$) with $L=1024, N=8192$ and $L=2048, N=16384$, for which $\Delta x=1.25\times 10^{-1}$. (It was checked experimentally that this was enough to satisfy the hypotheses (H1) and (H2) for both systems.) The evolution of the numerical solution is monitored up to a final time $T=800$.

\subsection{Linear stability}

Before showing how the previous methods behave when integrating the space discretizations of NLS equation or Boussinesq system, i.e. (\ref{nlsdis}) or (\ref{sdbb}), we first notice that the linear part of both systems have the form
\begin{eqnarray}
\dot{\left( \begin{array}{cc} P \\ Q \end{array} \right)}=  \left( \begin{array}{cc} 0 & M_1 \\ M_2 & 0 \end{array} \right) \left( \begin{array}{c} P \\ Q \end{array} \right),
\label{mpq}
\end{eqnarray}
for some matrices $M_1$ and $M_2$, where the eigenvalues of the whole matrix in (\ref{mpq}) are imaginary.

In the case of NLS equation, those eigenvalues will be like  $i \lambda^2$ where $i \lambda$ is each of the eigenvalues of $B_N$. In the case of Boussinesq system, they will have the form
$$i \lambda \sqrt{\frac{(1-a \lambda^2)(1-c \lambda^2)}{(1+b\lambda^2)(1+d \lambda^2)}},$$
so when $b=d>0$ and $|\lambda|$ grows, the modulus of the corresponding eigenvalues grow in the same way. Therefore, Boussinesq system is much less stiff than NLS equation.

In any case, after a change of variables in $P$ and in $Q$, the PLMM applied to (\ref{mpq}) will behave as that applied to
\begin{eqnarray}
\dot{\left( \begin{array}{cc} p\\ q \end{array} \right)}=  \left( \begin{array}{cc} 0 & \gamma \\ \delta & 0 \end{array} \right) \left( \begin{array}{c} p \\ q \end{array} \right), \mbox{ whenever } \gamma \delta <0.
\nonumber
\end{eqnarray}
In such a case,
$$
\rho_p(E) p_n= \Delta t \gamma \sigma_p(E) q_n, \quad  \rho_q(E) q_n= \Delta t \delta  \sigma_q(E) p_n,
$$
and so
$$
\rho_p(E) \rho_q(E) p_n= \Delta t^2 \gamma \delta \sigma_p(E) \sigma_q(E) p_n,
$$
and a similar equation holds for $q_n$. Then the size of the roots of
\begin{eqnarray}
r_{\Delta t}(x)=\rho_p(x)\rho_q(x)-\Delta t^2 \gamma \delta \sigma_p(x) \sigma_q(x)
\label{rh}
\end{eqnarray}
will determine how $\{p_n\}$ and   $\{q_n\}$ grow with $n$.

\begin{enumerate}
\item
In the case that the PLMM is symmetric with no common roots of $\rho_p$ and $\rho_q$ except $x_1=1$, c.f. \cite{CH}, for small enough $\Delta t$, all the roots of $r_{\Delta t}$ also have unit modulus. This comes from the fact that, whenever $z$ is a root of $r_{\Delta t}$, $\bar{z}$ and $z^{-1}$ are also roots of $r_{\Delta t}$ and
\begin{enumerate}
\item[(i)] near each single root of $\rho_p(x) \rho_q(x)$, just one single root of $r_{\Delta t}(x)$ can turn up for small enough $\Delta t$,
\item[(ii)] near the double root $x_1=1$ of $\rho_p(x) \rho_q(x)$, the roots of $r_{\Delta t}$ maybe real and inverse of each other or conjugate of unit modulus. Denoting them by $x_{1,\Delta t}=1+\eta(\Delta t)$, substituting in (\ref{rh}) and using that $\rho_p'(1)=\sigma_p(1)$, $\rho_q'(1)=\sigma_q(1)$ (due to consistency), it follows that
    $$
    \eta(\Delta t)^2 \approx \gamma \delta \Delta t^2.$$
    As $\gamma \delta <0$, $\eta(\Delta t) \approx \pm i \sqrt{-\gamma \delta} \Delta t$ and so $x_{1,\Delta t}$ also have unit modulus and are different.
\end{enumerate}
As a conclusion, all symmetric PLMMs are linearly stable in a neighbourhood of the origin in the imaginary axis.
\item
For NSPLMM2 method, $r_{\Delta t}$ in (\ref{rh}) has a root of modulus $>1$ which behaves as $1+O(\Delta t^5)$. Then, the region of stability does not contain any neighbourhood of the origin in the imaginary axis.
\item
For the third-order Adams method, the region of stability contains a neighbourhood of the origin in the imaginary axis, as shown in \cite{HLW}.
\end{enumerate}

\subsection{Approximation of solitary waves}
\label{sec31}

\begin{figure}[htbp]
\centering
\subfigure[]
{\includegraphics[width=6.2cm]{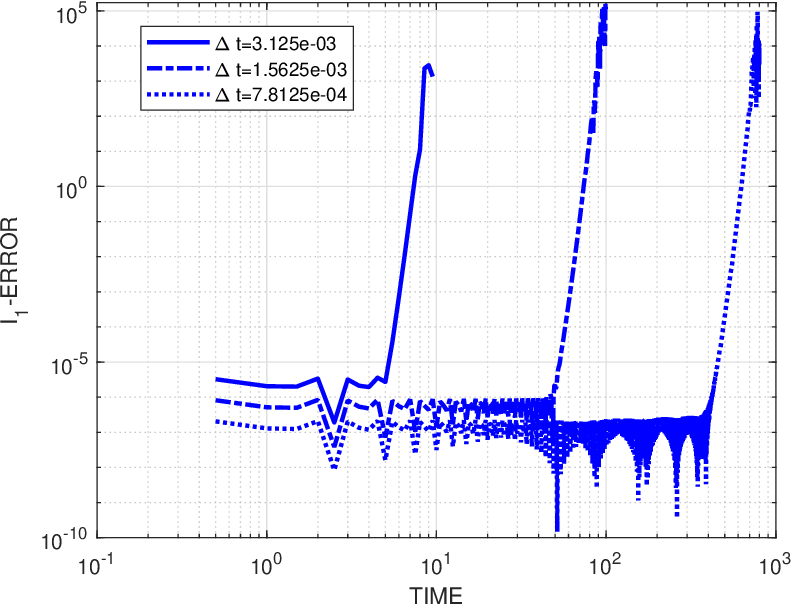}}
\subfigure[]
{\includegraphics[width=6.2cm]{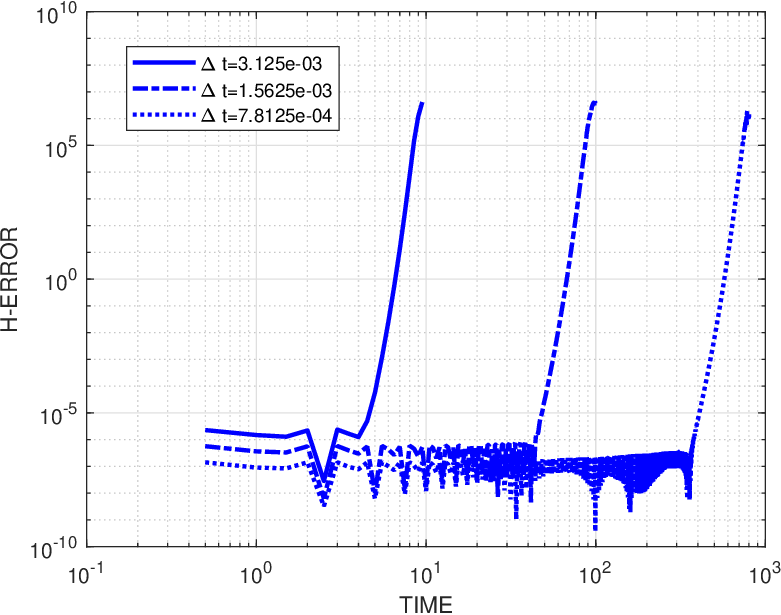}}
\caption{Error in the invariants against time with NSPLMM2  w.r.t. a solitary wave (\ref{sw}) {of the NLS equation (\ref{nls2})} with $\sigma=1$ and parameters  $a=\lambda_{0}^{2}=1, x_{0}=-400, \theta_{0}=\pi/4$. (a) Mass error. (b) Energy error. $\Delta x=1.25\times 10^{-1}$.}
\label{FIG1}
\end{figure}

As mentioned in the introduction, two groups of experiments will be made for each case study. The first one aims at illustrating the previous results with numerical examples corresponding to the approximation of solitary wave solutions.

For the first experiments concerning the NLS equation (\ref{nls2}), we consider a soliton solution (\ref{sw}) for the case $\sigma=1$ and with parameters  $a=\lambda_{0}^{2}=1, x_{0}=-400, \theta_{0}=\pi/4$.

\begin{figure}[htbp]
\centering
\subfigure[]
{\includegraphics[width=6.2cm]{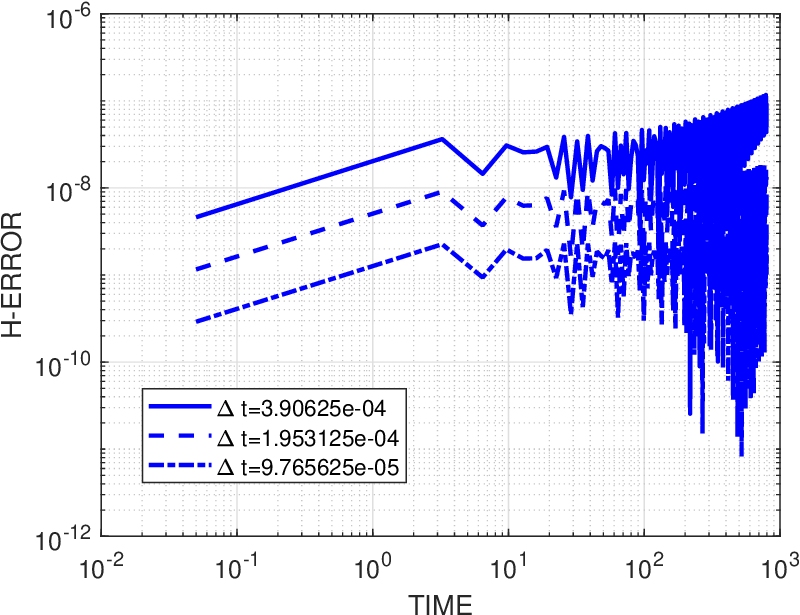}}
\subfigure[]
{\includegraphics[width=6.2cm]{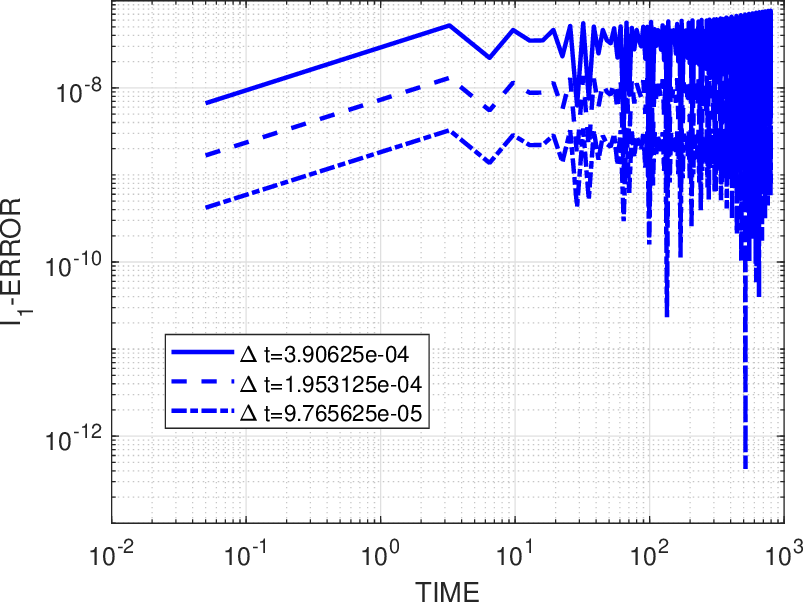}}
\caption{Error in the invariants against time with NSPLMM2 w.r.t. a solitary wave (\ref{sw}) {of the NLS equation (\ref{nls2})} with $\sigma=1$ and parameters  $a=\lambda_{0}^{2}=1, x_{0}=-400, \theta_{0}=\pi/4$. (a) Mass error. (b) Energy error. $\Delta x=1.25\times 10^{-1}$.}
\label{FIG2}
\end{figure}

The time behaviour of the error (in log-log scale) with respect to the mass $l I_{1,N}/N$ and the Hamiltonian $l H_{N}/N$, conserved quantities of solutions of the semidiscrete system (\ref{nlsdis}), given by NSPLMM2 for several step sizes, is shown in Figure \ref{FIG1}. The errors are observed to show order $2$ when the time-stepsize diminishes and to behave oscillatory (but in a bounded way) till a certain time in which they begin to grow exponentially. The time in which the type of behaviour changes grows a lot when $\Delta t$ diminishes and thus it seems to be associated to the lack of linear stability of this method around the origin at the imaginary axis, which was commented in the previous subsection. When the time step is further reduced, that exponential growth is not observed any more and the growth of error on the invariants is bounded till a certain time $t\approx 100$ (for $\Delta t=3.90625e-4$) and then begins to grow a bit less than linearly, as Figure \ref{FIG2} shows. This is in accordance with Remark \ref{rem1}, which explains that the term in the error associated to $\Delta t^2$ is bounded and that associated to $\Delta t^3$ may grow linearly.

\begin{figure}[htbp]
\centering
\subfigure[]
{\includegraphics[width=6.2cm]{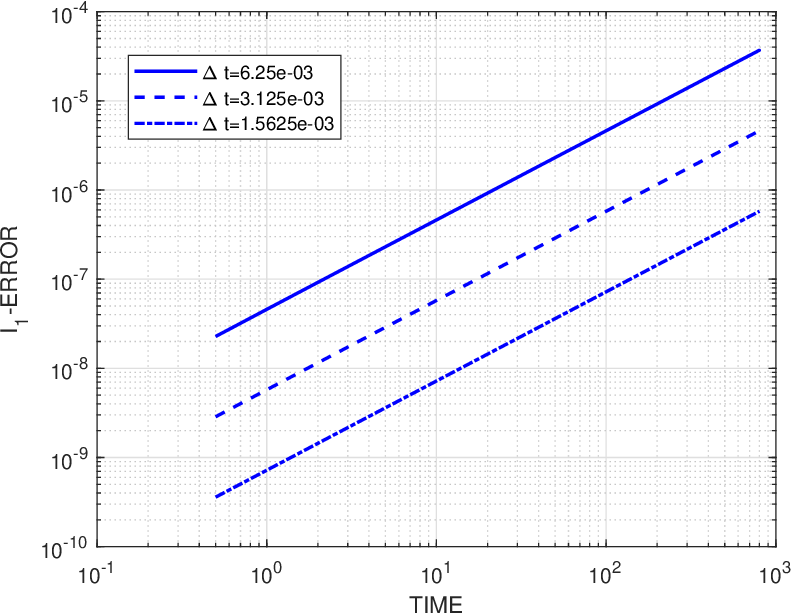}}
\subfigure[]
{\includegraphics[width=6.2cm]{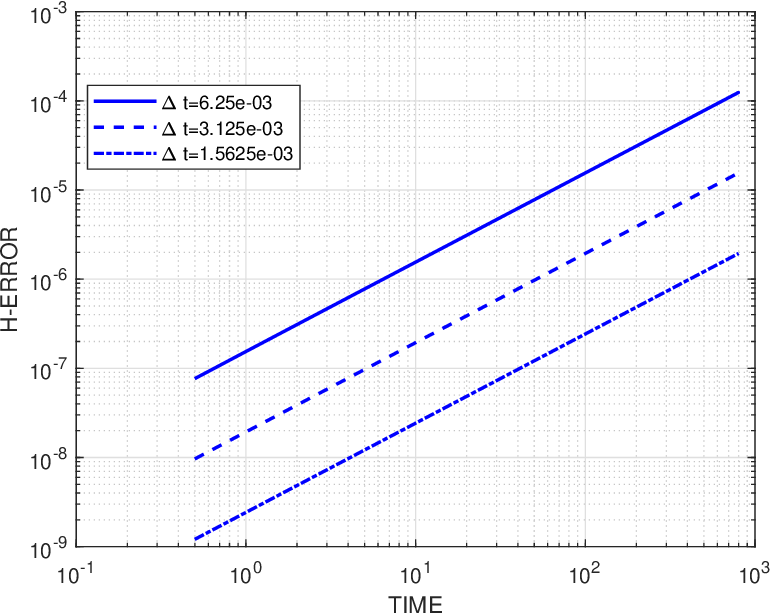}}
\subfigure[]
{\includegraphics[width=6.2cm]{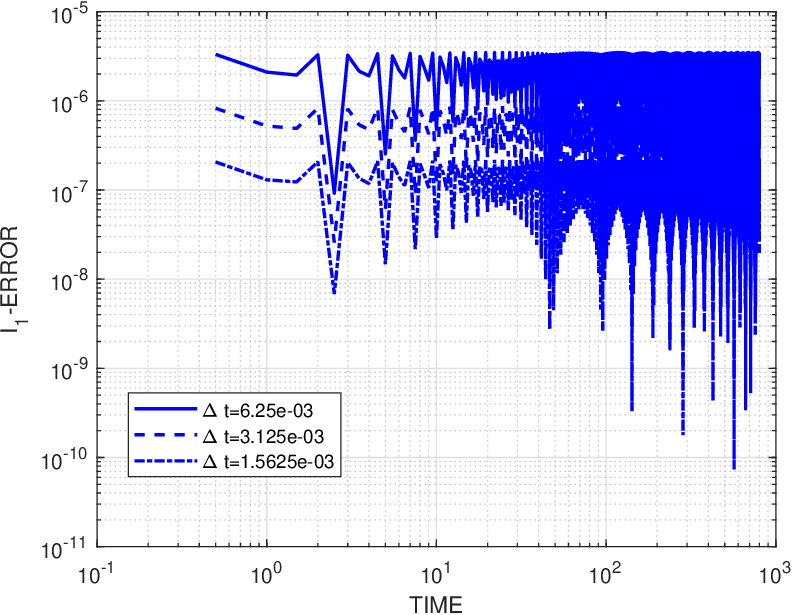}}
\subfigure[]
{\includegraphics[width=6.2cm]{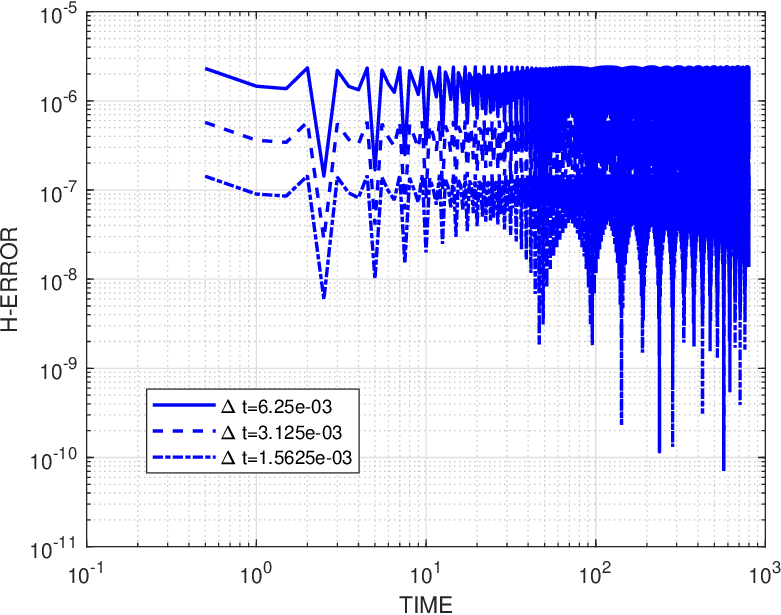}}
\caption{Error in the invariants against time w.r.t. a solitary wave (\ref{sw}) {of the NLS equation (\ref{nls2})} with $\sigma=1$ and parameters  $a=\lambda_{0}^{2}=1, x_{0}=-400, \theta_{0}=\pi/4$. (a), (b) Mass and energy error for NSNPLMM3. (c), (d) Mass and energy error for SPLMM2. $\Delta x=1.25\times 10^{-1}$.}
\label{FIG3}
\end{figure}

 The previous experiments can be contrasted with the results corresponding to the more linearly stable NSNPLMM3 and SPLMM2, and shown in Figure \ref{FIG3} for the biggest time stepsizes. Observe that, while NSNPLMM3 develops a linear error growth in the evolution of the invariants, these errors are bounded in the case of SPLMM2 up to the final time of simulation, confirming the results in Theorem \ref{th1} and Remark \ref{rem1}.

 Note also that the relative equilibrium condition (\ref{RE})
will ensure that the error with respect to the second invariant $I_{2,N}$ shows a similar behaviour. For this reason the corresponding results will not be shown here.


In the following experiments we show the results for the case of the Boussinesq family of systems
(\ref{bb})  and the simulation of solitary wave solutions, cf. (\ref{simh}), (\ref{sw3}).
By way of illustration, we consider the parameters
\begin{eqnarray}
\theta^{2}=9/11, b=d=\frac{1}{2}(\theta^{2}-1/3), c=2/3-\theta^{2}, a=0,
\label{param}
\end{eqnarray}
corresponding to the so-called Bona-Smith systems, \cite{DougalisM2008}, and the solitary wave
with $x_{0}=-400$ and speed $c_{s}=1.4434$.

\begin{figure}[htbp]
\centering
\subfigure[]
{\includegraphics[width=6.2cm]{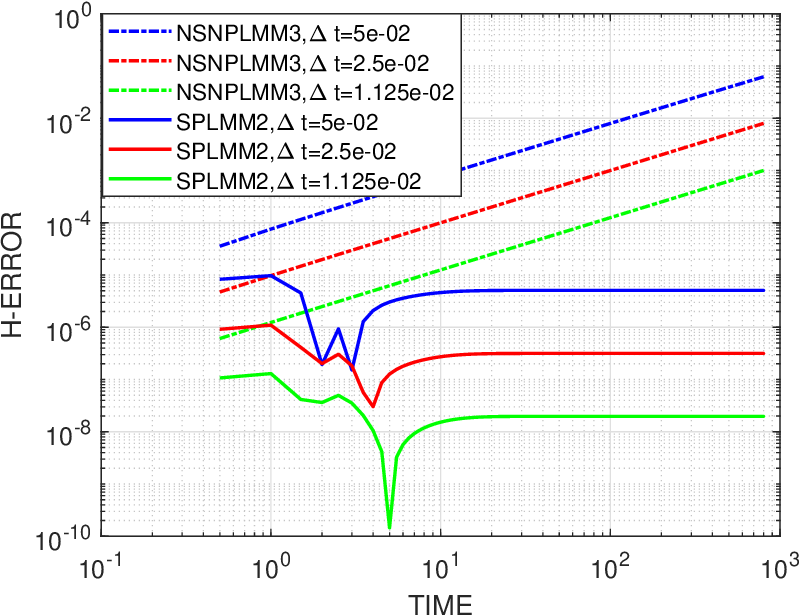}}
\subfigure[]
{\includegraphics[width=6.2cm]{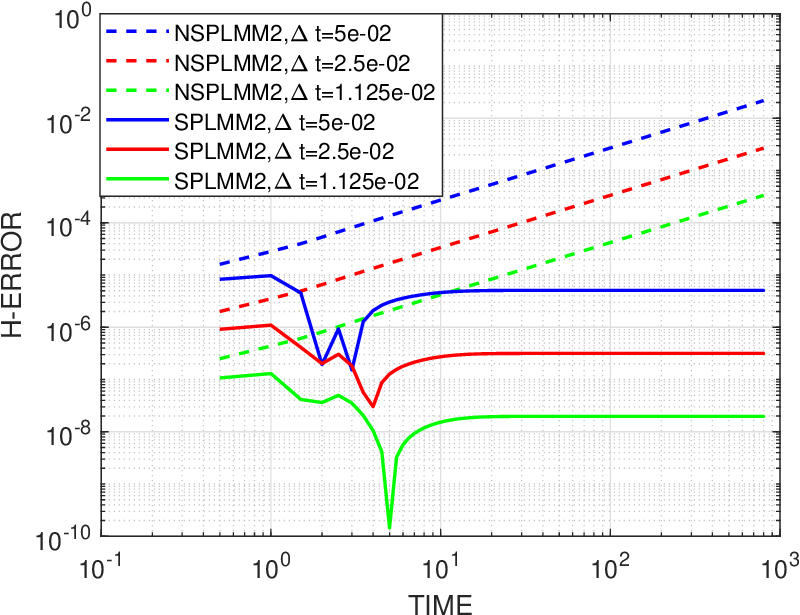}}
\caption{Error in the Hamiltonian against time  w.r.t. a solitary wave {of Boussinesq system (\ref{bb})} with $\theta^{2}=9/11, b=d=\frac{1}{2}(\theta^{2}-1/3), c=2/3-\theta^{2}, a=0$ . $\Delta x=1.25\times 10^{-1}$. (a) Dashed lines: NSNPLMM3; solid lines: SPLMM2. (b) Dashed lines: NSPLMM2; solid lines: SPLMM2.}
\label{FIG4}
\end{figure}

\begin{figure}[htbp]
\centering
\subfigure[]
{\includegraphics[width=0.5\columnwidth]{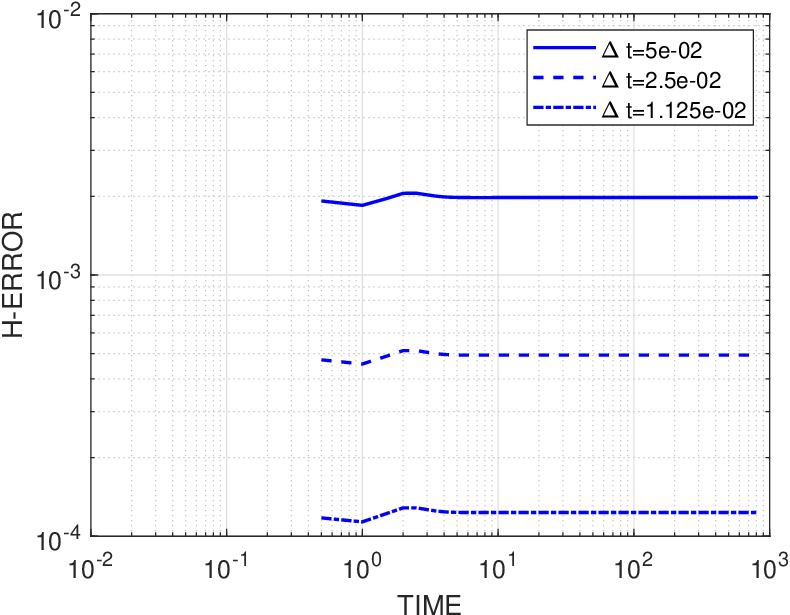}}
\caption{Error in the Hamiltonian against time  w.r.t. a solitary wave {of Boussinesq system (\ref{bb})} with $\theta^{2}=9/11, b=d=\frac{1}{2}(\theta^{2}-1/3), c=2/3-\theta^{2}, a=0$. SPLMM2 with starting values obtained from the implicit midpoint rule.}
\label{FIG4bis}
\end{figure}

When the starting values are taken as exact, the time evolution of the error in the Hamiltonian $H_{N}$ with respect to this solitary wave and for the three time integrators, is shown, in log-log scale, in Figure \ref{FIG4}. In Figure \ref{FIG4}(a) we compare SPLMM2 with NSNPLMM3, and in  Figure \ref{FIG4}(b) the comparison is between NSPLMM2 and SPLMM2. In Figure \ref{FIG4}(a), we observe the linear growth of the errors corresponding to NSNPLMM3, against the bounded behaviour of those for SPLMM2. It is remarkable that the error with SPLMM2 diminishes like $O(\Delta t^4)$ in spite of being a method of order $2$, as stated in part (iii) of Theorem \ref{th2}. On the other hand, the evolution of the error given by NSPLMM2, shown in Figure \ref{FIG4}(b), reveals a linear growth with time from the very beginning which, in constrast to Figure \ref{FIG2} for NLS equation, comes from the fact that the term of the error associated to $\Delta t^2$ vanishes for the solitary wave solution of Boussinesq system due to the same arguments as with SPLMM2. As for the term in $\Delta t^3$ associated to the smooth part of the numerical solution, it  grows linearly according to Remark \ref{remark32}. This explains both the linear growth from the very beginning in this case and the order $3$ when the time stepsize diminishes.

We notice that in case the starting values are calculated with order $2$ using the implicit midpoint rule,  the term of the error in the Hamiltonian in $\Delta t^2$ associated to the non-smooth part of the numerical solution does not vanish any more, and thus order $2$ is observed, as shown in Figure \ref{FIG4bis} for SPLMM2 method.

As in the previous example, the relative equilibrium condition (\ref{RE2b}) or, equivalently, (\ref{RE2}), implies that the errors with respect to the quadratic invariant $I_{N}$ evolve in a similar way. We also computed the time behaviour of the errors with respect to the linear invariants $M_{1,N}, M_{2,N}$, and they are below machine precision in all the cases; therefore, the three schemes virtually preserve the two quantities, as expected by Remark \ref{remlin}.

\subsection{Approximation of other localized solutions: perturbations of solitary waves and resolution property}
\label{sec32}
A second question considered at this point is if  the good behaviour of the symmetric method SPLMM2 can be extended to the simulation of other localized waves. This will be explored in this section with several numerical examples. In the first one we study, by computational means, the performance of SPLMM2 when simulating small perturbations of solitary wave solutions of (\ref{nls}). To this end, we consider the cubic case ($\sigma=1$) and the soliton solution profile $u_{0}=(p_{0},q_{0})$ given by (\ref{sw}) with
$a=\lambda_{0}^{2}=1, x_{0}=-400, \theta_{0}=\pi/4$ at $t=0$. This is slightly perturbed to generate a profile $\widetilde{u}_{0}=(\widetilde{p}_{0},\widetilde{q}_{0})$ with
\begin{eqnarray}
\widetilde{p}_{0}=A_{1}p_{0},\quad \widetilde{q}_{0}=A_{2}q_{0},\label{swpert}
\end{eqnarray}
for perturbation factors $A_{j}, j=1,2$. Then the numerical solution given by SPLMM2 to approximate (\ref{nlsdis}) with (\ref{swpert}) as initial condition,  with a starting procedure of the same order, is monitored. For the values $A_{1}=A_{2}=1.05$, the initial perturbed wave evolves to a solution with two components: one is a waveform which tends asymptotically to a modified solitary wave, with speed and amplitude close to those of the original, unperturbed one (cf. Figure \ref{cd_FIG6a}). The second component may include dispersive and small amplitude waves, and tends to zero as time evolves (cf. Figure \ref{cd_FIG6b}). This illustrates the asymptotic stability of the solitary wave, \cite{CuccagnaP}.

\begin{figure}[htbp]
\centering
\subfigure
{\includegraphics[width=4.1cm]{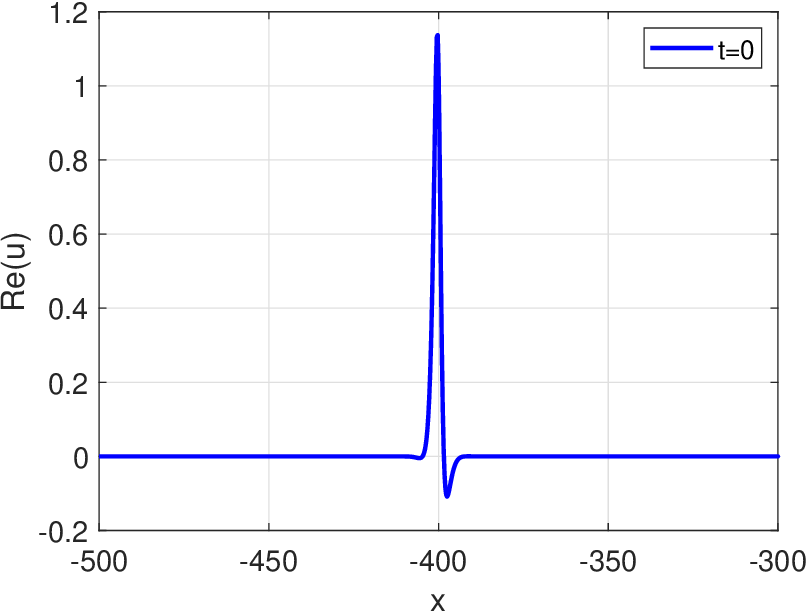}}
\subfigure
{\includegraphics[width=4.1cm]{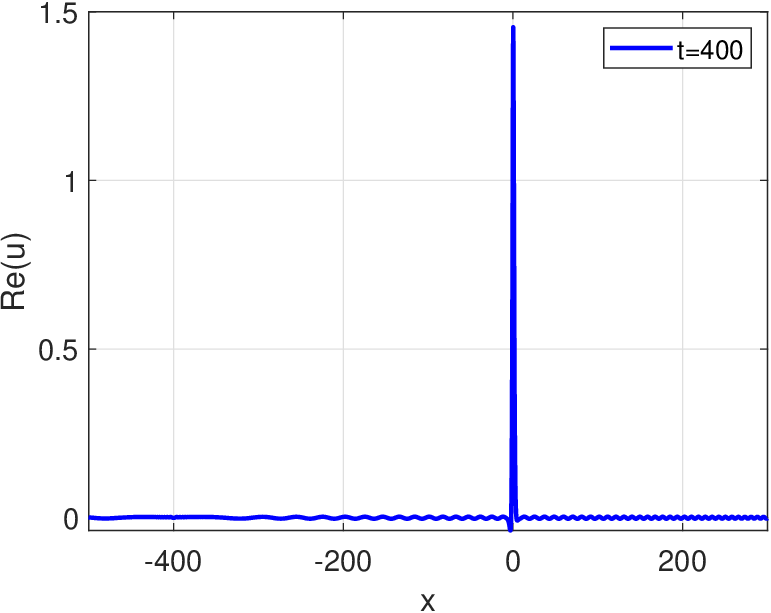}}
\subfigure
{\includegraphics[width=4.1cm]{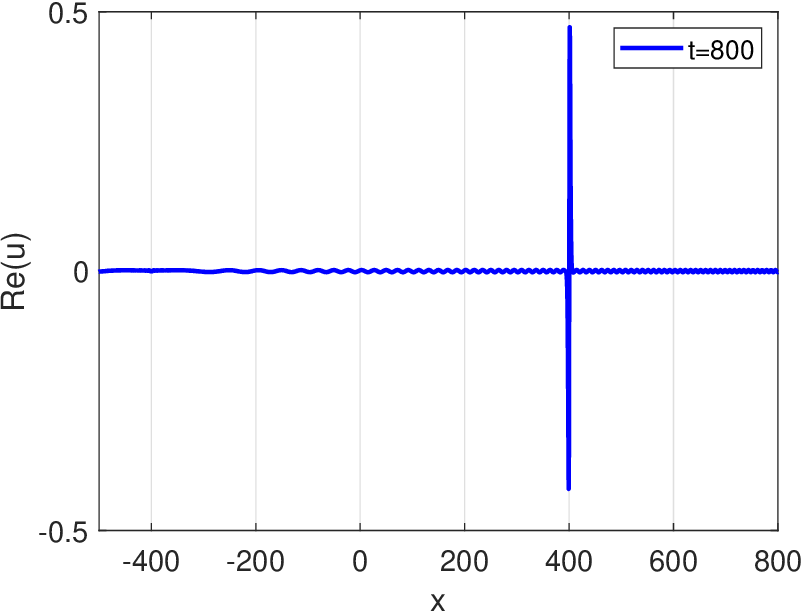}}
\subfigure
{\includegraphics[width=4.1cm]{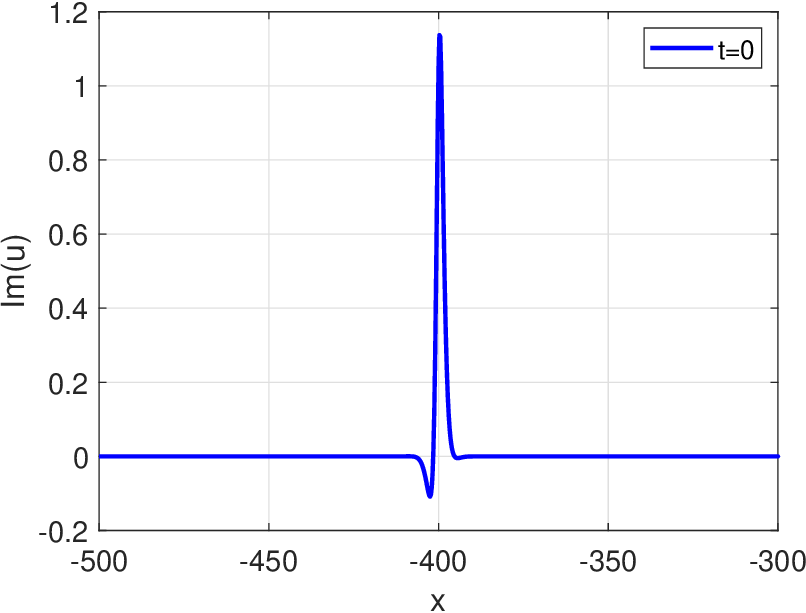}}
\subfigure
{\includegraphics[width=4.1cm]{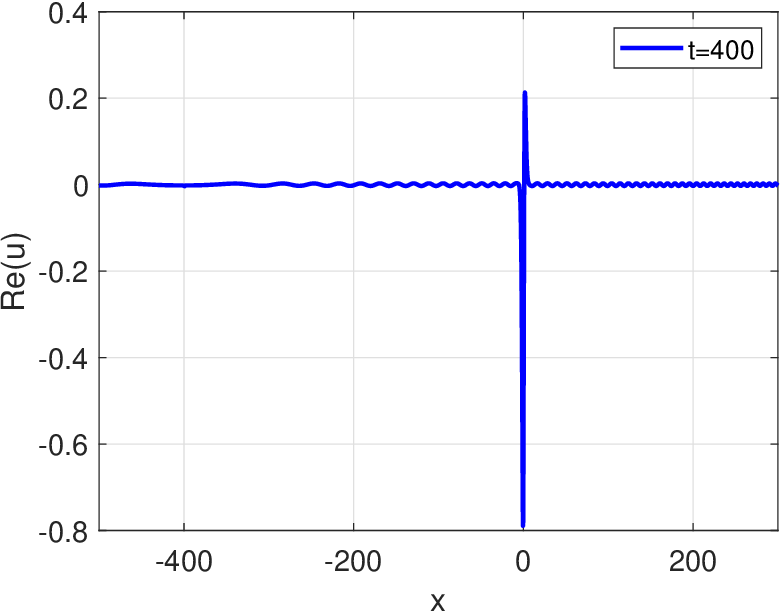}}
\subfigure
{\includegraphics[width=4.1cm]{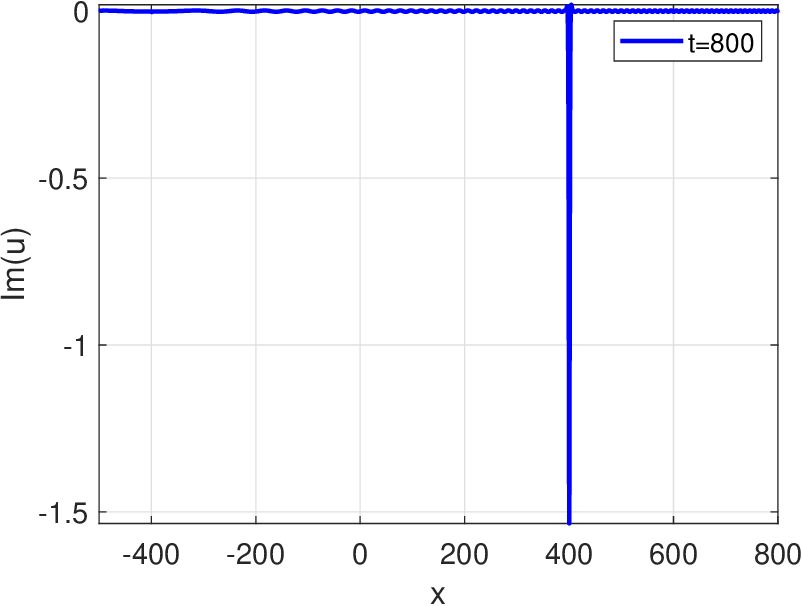}}
\caption{$p,q$ components of the numerical solution with SPLMM2 from a slight perturbation (\ref{swpert}) of a solitary wave
{of the NLS equation} (\ref{nls}) with $A_{1}=1.05, A_{2}=1.05$ at $t=0, 400, 800$.}
\label{cd_FIG6a}
\end{figure}
\begin{figure}[htbp]
\centering
\subfigure
{\includegraphics[width=6.2cm]{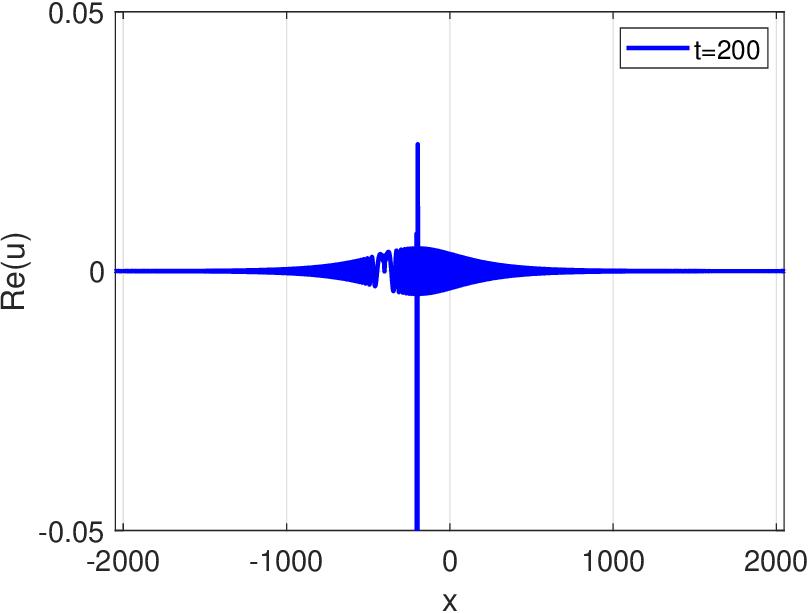}}
\subfigure
{\includegraphics[width=6.2cm]{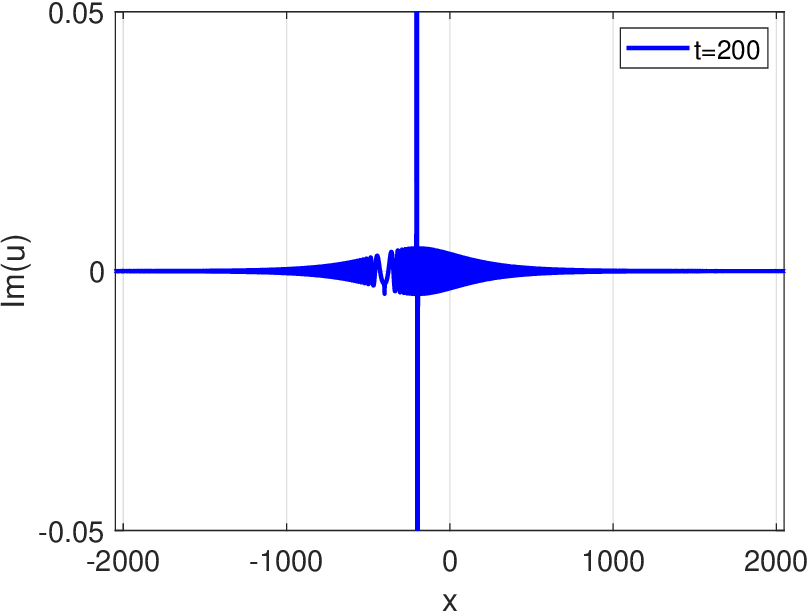}}
\caption{$p,q$ components of the numerical solution with SPLMM2 from a slight perturbation (\ref{swpert}) of a solitary wave {of the NLS equation} (\ref{nls}) with $A_{1}=1.05, A_{2}=1.05$, at $t=200$ (magnified).}
\label{cd_FIG6b}
\end{figure}

In Figure \ref{cd_FIG7} the time behaviour of the error in mass, momentum, and energy quantities for SPLMM2, with several time step sizes, suggests that the good performance of the symmetric method with respect to the soliton solutions, depicted in Section \ref{sec31}, can be extended to the simulations of close localized solutions obtained from small perturbations.
\begin{figure}[htbp]
\centering
\subfigure[]
{\includegraphics[width=6.2cm]{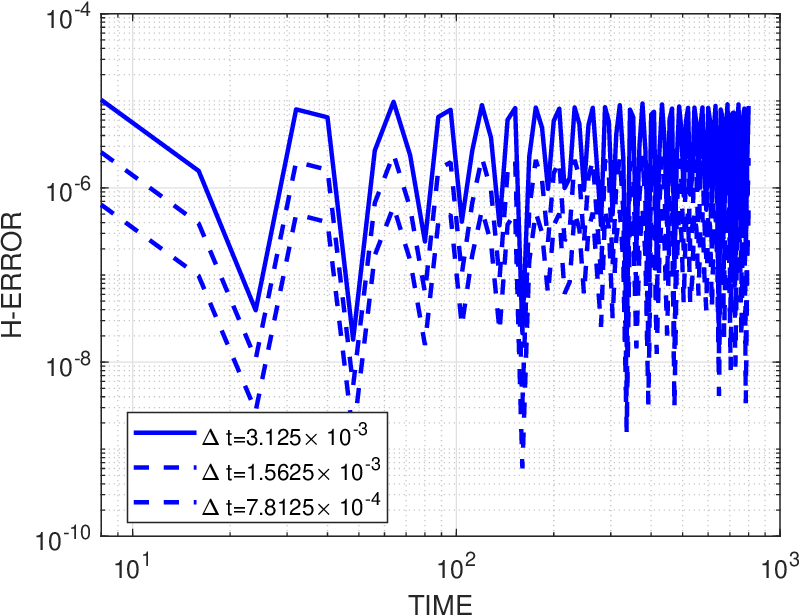}}
\subfigure[]
{\includegraphics[width=6.2cm]{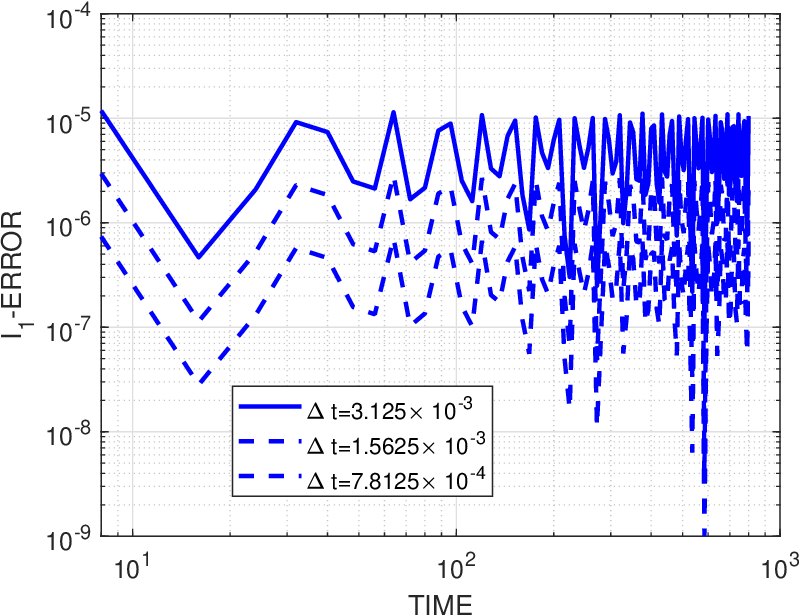}}
\subfigure[]
{\includegraphics[width=6.2cm]{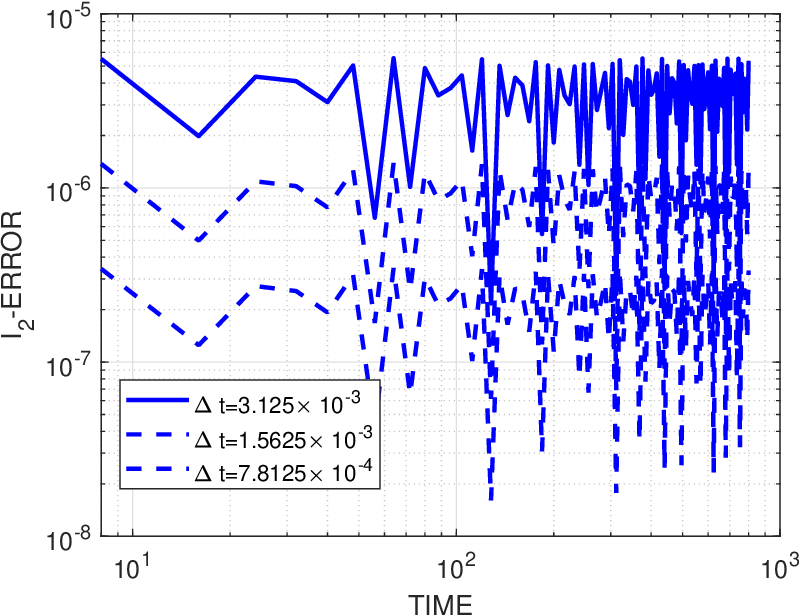}}

\caption{Error in the invariants against time {when integrating a slight perturbation (\ref{swpert}) of a solitary wave of NLS equation with $A_{1}=1.05, A_{2}=1.05$} with SPLMM2: $\Delta x=1.25\times 10^{-1}$. (a) Mass; (b) momentum; (c) energy.}
\label{cd_FIG7}
\end{figure}

A second example is concerned with the so-called resolution property for the systems (\ref{bb}) with $b=d$, see e.~g. \cite{DDS,DougalisM2008} and references therein. This is described as the property that some localized initial data can have to evolve  towards the formation of a train of solitary waves plus dispersive, small-amplitude tails.
\begin{figure}[htbp]
\centering
\subfigure[]
{\includegraphics[width=6.2cm]{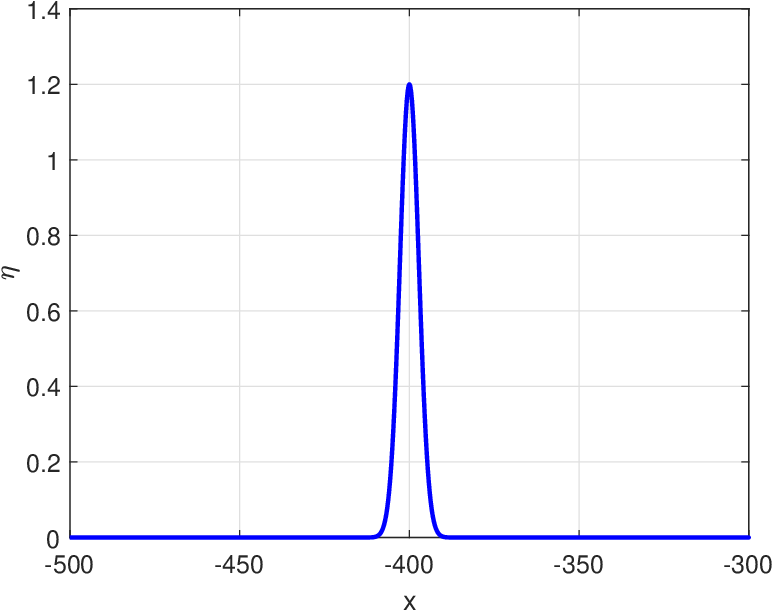}}
\subfigure[]
{\includegraphics[width=6.2cm]{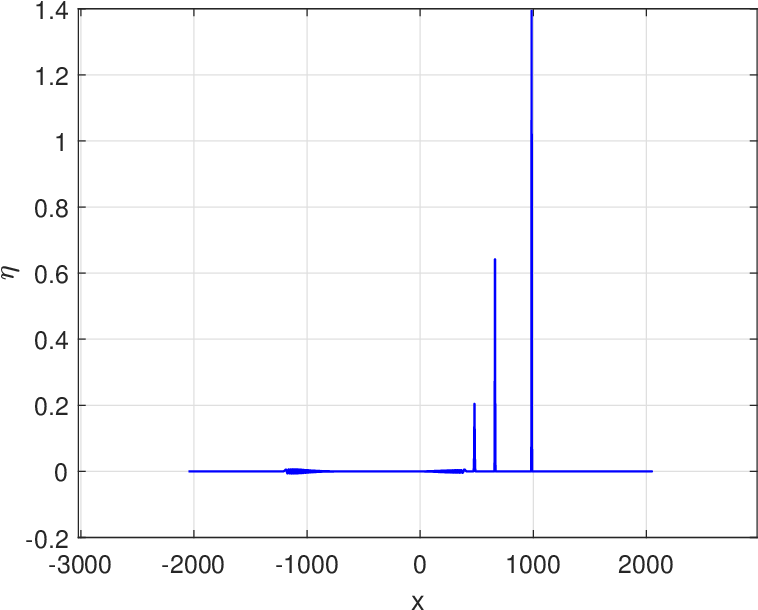}}
\subfigure[]
{\includegraphics[width=6.2cm]{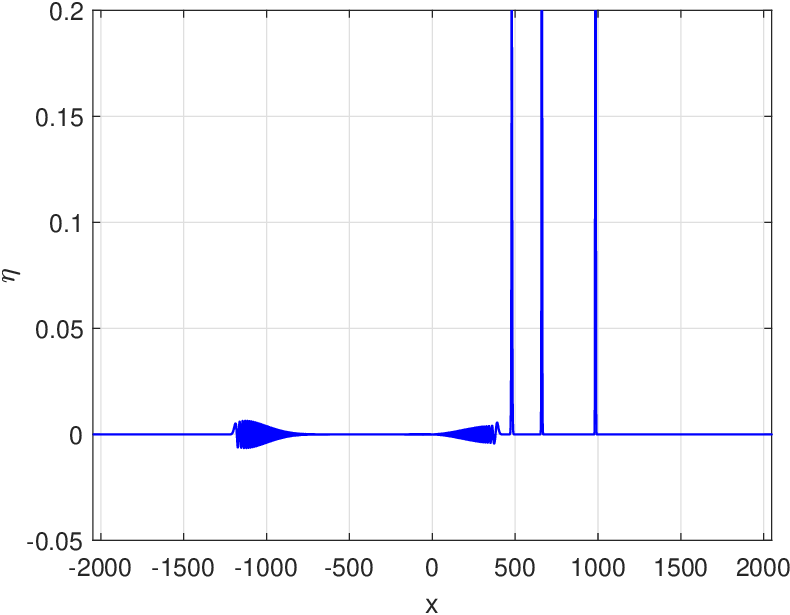}}
\subfigure[]
{\includegraphics[width=6.2cm]{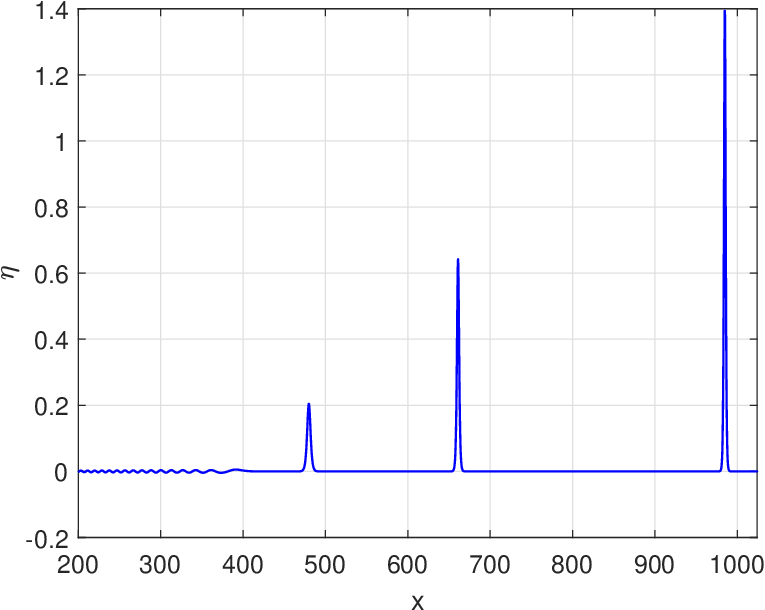}}
\caption{Resolution property with SPLMM2 {when integrating Boussinesq system (\ref{bb})} with initial condition (\ref{Gaussian}): $\Delta x=1.25\times 10^{-1}, \Delta t=1.5625\times 10^{-3}$. (a) $\eta$ component of an initial Gaussian profile; (b) numerical solution at $t=800$; (c),(d) magnifications of (b).}
\label{cd_FIG8}
\end{figure}
In Figure \ref{cd_FIG8}(a) we show a Gaussian profile
\begin{eqnarray}
&&\eta_{0}(x)=Ae^{-B(x-x_{0})^{2}},\; w_{0}(x)=C\eta_{0}(x),\nonumber\\
&&A=1.2, B=6.42\times 10^{-2}, C=0.87, x_{0}=-400,\label{Gaussian}
\end{eqnarray}
which is taken as initial condition for the semidiscrete system (\ref{sdbb}) with $a,b,c,d$ in (\ref{param}). The problem is integrated with SPLMM2 (with a starting procedure of the second order) and in Figure \ref{cd_FIG8}(b) the numerical solution at $t=800$ is shown. This consists of a train of solitary waves traveling to the right along with two groups of  dispersive tails, one following the train and the second one traveling in the opposite direction.
\begin{figure}[htbp]
\centering
\subfigure[]
{\includegraphics[width=6.2cm]{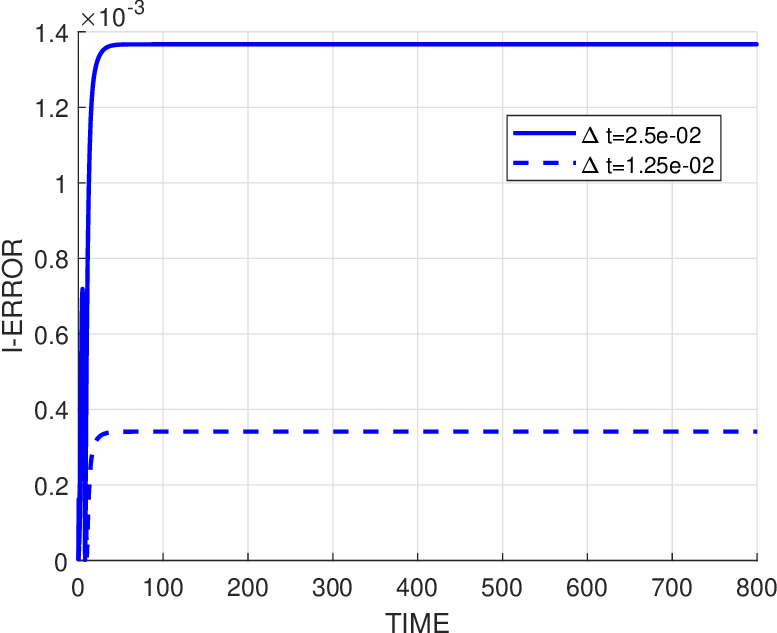}}
\subfigure[]
{\includegraphics[width=6.2cm]{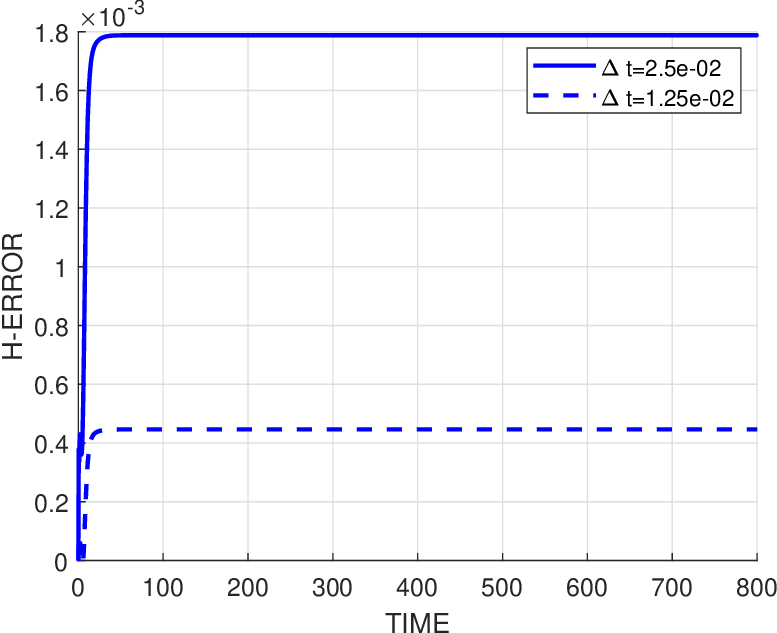}}
\caption{Error in the invariants against time with SPLMM2 when integrating Boussinesq system (\ref{bb}) with initial condition (\ref{Gaussian}): $\Delta x=1.25\times 10^{-1}$.}
\label{cd_FIG9}
\end{figure}
During the simulation, we measured the time evolution of the error in the last two invariants $H_{N}$ and $I_{N}$, for two time-step sizes. From the results given in Figure \ref{cd_FIG9}, we observe that the good behaviour of the method still holds.

\section*{Acknowledgments}
This research has been supported by Ministerio de Ciencia e Innovaci\'on project PID2023-147073NB-I00.

 \section*{Conflict of interests} The author  has no conflicts of interest to declare.
	 
	 \section*{Data availability statement} 
	 Data sharing is not applicable to this article as no new data were created or analyzed in this study.

\bibliographystyle{plain}

\end{document}